\documentclass[12pt,a4paper]{article}
\usepackage[top=3.05cm, bottom=3.05cm, left=3cm, right=3cm]{geometry}
\usepackage[latin1]{inputenc}
\usepackage{csquotes}
\usepackage{amsmath}
\usepackage{amsthm}
\usepackage{amsfonts}
\usepackage{amssymb}
\usepackage{graphics}
\usepackage{float}
\usepackage[hang,flushmargin]{footmisc} %nonindent footnote

\usepackage{amsmath,amsthm,amssymb,graphicx, multicol, array}
\usepackage{enumerate}
\usepackage{enumitem}
\usepackage{xcolor}
\usepackage[pagebackref,bookmarks,colorlinks,breaklinks]{hyperref}

\hypersetup{linkcolor=blue,citecolor=red,filecolor=blue,urlcolor=blue} 

\usepackage{tabto}

\numberwithin{equation}{section}

\usepackage{tikz}
\usepackage{pgfplots}

\usetikzlibrary{matrix}
\usepackage{mathrsfs}

\DeclareMathOperator{\ord}{ord}

\newtheorem{thm}{Theorem}[section]
\newtheorem{lem}{Lemma}[section]
\newtheorem{theorem}{Theorem}[section]
\newtheorem{lemma}{Lemma}[section]

\newtheorem{exa}{Example}[section]

\newtheorem{cor}{Corollary}[section]

\newtheorem{exe}{Exercise}[section]

\newcommand{\Q}{\mathbb{Q}}

\newcommand{\F}{\mathbb{F}}

\usepackage{fancyhdr}
\usepackage{titletoc}

\titlecontents{section}
[0pt]                                               % left margin
{}%
{\contentsmargin{0pt}                               % numbered entry format
	\thecontentslabel\enspace%
	\normalsize}
{\contentsmargin{0pt}\normalsize}                        % unnumbered entry format
{\titlerule*[.5pc]{.}\contentspage}                 % filler-page format (e.g dots)
[] 
\title{Small Prime $k$th Power Residues and Nonresidues in Arithmetic Progressions}
\date{}
\author{N. A. Carella}

\begin{document}
%\doublespacing
\thispagestyle{empty}
\date{}

\maketitle
\begin{abstract}
Let $p$ be a large odd prime, let $x=\log p)(\log\log p)^{3+\varepsilon}$ and let $q\ll\log\log p$ be an integer, where $\varepsilon>0$ is a small number.  
This note proves the existence of small prime quadratic residues and small prime quadratic nonresidues in the arithmetic progression $a+qm\ll x$, with relatively prime $1\leq a<q$, unconditionally. The same results are generalized to small prime $k$th power residues and nonresidues, where $k\mid p-1$ and $k\ll\log\log p$. \let\thefootnote\relax\footnote{\today  
	}
\end{abstract}

\tableofcontents

%SSSSSSSSSSSSSSSSSSSSSSSSSSSSSSSSSSSSSSSSSSSSSS
%SSSSSSSSSSSSSSSSSSSSSSSSSSSSSSSSSSSSSSSSSSSSSS
%SSSSSSSSSSSSSSSSSSSSSSSSSSSSSSSSSSSSSSSSSSSSSS
%SSSSSSSSSSSSSSSSSSSSSSSSSSSSSSSSSSSSSSSSSSSSSS
%SSSSSSSSSSSSSSSSSSSSSSSSSSSSSSSSSSSSSSSSSSSSSS
\section{Introduction}\label{S9982QN}\hypertarget{S9982QN}
For an odd prime $p$ and a pair of integers $k\geq2$ and $n\ne0,1$, the congruence equation $x^k-n\equiv 0 \bmod p$ is solvable if and only if the integer $n$ is a \textit{kth power residue}. Otherwise $n$ is a \textit{kth power nonresidue}. The literature for $k=2$, better known as quadratic residues and quadratic nonresidue, is vast, see \cite{MT2021} for a survey. However, very few results on quadratic residues and quadratic nonresidue in arithmetic progressions are known, see \cite{GA2006}, \cite{WS2013}, \cite{PP2018} for some details. Likewise, the literature for $k>2$ is scarce, there are very few papers, see \cite{HR1974}. \\

The most recent results in \cite{GA2006} confirm the existence of prime quadratic residues $n=4m+1<p$ and $n=4m+3<p$ modulo $p>37$ but gives no details on the size of these quadratic residues. The result below claims the existence of small prime quadratic residues $n=a+qm$ for any arithmetic progression with small relatively prime parameters $a$, $q$ and large $p$. 

\begin{theorem}\label{thm9982QR.100} \hypertarget{thm9982QR.100} Let $\varepsilon>0$ be a small real number. Let \(p\geq 2\) be a large prime, let $ x=(\log p)(\log\log p)^{3+\varepsilon}$ and let $q=(\log\log p)$. If $1 \leq a <q$ is a pair of relatively prime integers, then there exists a prime quadratic residue in the arithmetic progression
	\begin{equation}\label{eq9982QR.100}
		n=a+qm\ll x.   
	\end{equation}
\end{theorem}
Similarly, the results in \cite{PP2018} confirm the existence of prime quadratic nonresidues $n=4m+1<p$ and $n=4m+3<p$ modulo $p$ but gives no details on the size of these quadratic residues. Further, an earlier result in {\color{red}\cite[Theorem 1]{HR1974}} specifies the upper bound $n\ll p^{2/5}q^{5/2}$. The result below claims the existence of small quadratic nonresidues $n=a+qm$ for any arithmetic progression with small relatively prime parameters $a$, $q$ and large $p$.

\begin{theorem}\label{thm9982QN.100} \hypertarget{thm9982QN.100} Let $\varepsilon>0$ be a small real number. Let \(p\geq 2\) be a large prime, let $ x=(\log p)(\log\log p)^{3+\varepsilon}$ and let $q=(\log\log p)$. If $1 \leq a <q$ is a pair of relatively prime integers, then there exists a prime quadratic nonresidue in the arithmetic progression
	\begin{equation}\label{eq9982QN.100N}
		n=a+qm\ll x.   
	\end{equation}
\end{theorem}
Larger moduli $q=(\log\log p)^c$, where $c\geq0$ is an arbitrary constant, are possible. However, these are restricted by the limitation of the Siegal-Walfizs theorem and the value $x$ increases by a corresponding amount. Furthermore, exponentially large moduli $q\approx(\log p)^{1/2}$ conditional on the GRH or the Mongomerry conjecture, see {\color{red}\cite[Conjecture 13.9]{MV2007}}, and  increasing the value $x$ by a corresponding amount, are also possible, but the value $x$ increases by a corresponding exponential amount.\\

The results in {\color{red}\cite[Theorem 1]{HR1974}} claims that the least $k$th power nonresidue $n(p,q,a)=a+mq$ in arithmetic progressions satisfies $n(p,q,a)\ll p^{2/5}q^{5/2}$. A significant sharpening of this estimate is provided here, for both $k$th power residues and nonresidues. 

\begin{theorem}\label{thm9982KR.100} \hypertarget{thm9982KR.100} Let $\varepsilon>0$ be a small real number. Let \(p\geq 2\) be a large prime, let $ x=(\log p)(\log\log p)^{4+\varepsilon}$ and let $q\ll(\log\log p)$ be an integer. If $k\mid p-1$ is a small integer such that $k<x$ and $1 \leq a <q$ are relatively prime integers, then there exists a prime $k$th power residue in the arithmetic progression
		\begin{equation}\label{eq9982KR.100a}
			n=a+qm\ll x.\nonumber   
		\end{equation}

	In addition, the predicted number of small prime $k$th residues or nonresidues is 
	\begin{equation}\label{eq9982KR.100CD}
		N_k(x,q,a)=\delta(k,q,a)\cdot (\log p)(\log\log p)^{3+\varepsilon}+Error ,
	\end{equation}
	where $\delta(k,q,a)=c(k,q,a)/(k\varphi(q))\geq0$ is the density.
\end{theorem}
\begin{theorem}\label{thm9982KN.100} \hypertarget{thm9982KN.100} Let $\varepsilon>0$ be a small real number. Let \(p\geq 2\) be a large prime, let $ x=(\log p)(\log\log p)^{4+\varepsilon}$ and let $q\ll(\log\log p)$ be an integer. If $k\mid p-1$ is a small integer such that $k<x$ and $1 \leq a <q$ are relatively prime integers, then there exists a prime $k$th power nonresidue in the arithmetic progression
\begin{equation}\label{eq9982KN.100b}
	n=a+qm\ll x.\nonumber   
\end{equation}	
In addition, the predicted number of small prime $k$th residues or nonresidues is 
\begin{equation}\label{eq9982KN.100CD}
	N_k(x,q,a)=\delta(k,q,a)\cdot (\log p)(\log\log p)^{3+\varepsilon}+Error ,
\end{equation}
where $\delta(k,q,a)=c(k,q,a)/(k\varphi(q))\geq0$ is the density.
\end{theorem}
The density is defined by

\begin{equation}\label{eq9982KR.100CDK}
\delta(k,q,a)=\lim_{p\to\infty} \, \frac{\#\{a+qm=n < p: n\text{ is a }k\text{th prime quadratic residue}\}}{ p}.
\end{equation}
It includes the correction factor $c(k,q,a)$ which accounts for the  statistical dependencies between the primes $k$th power nonresidues, depending on $k\geq2$, $q\geq2$ and $a\geq1$. \\

The form of the counting function \eqref{eq9982KR.100CD} for $k\geq2$ suggests that the properties of being a prime, a $k$th power nonresidue and in an arithmetic progression are statistically independent events if and only if $c(k,q,a)=1$. On the other extreme, the counting function \eqref{eq9982KR.100CD} is finite if and only if $c(k,q,a)=0$, for example this event can occur whenever $\gcd(a,q)\ne1$. \\

The analysis is based on new characteristic functions for quadratic residues and quadratic nonresidues in finite fields introduced in \hyperlink{S9911Q}{Section} \ref{S9911Q}. 
Except for small changes the proof of \hyperlink{thm9982QR.100}{Theorem} \ref{thm9982QR.100}  is the same as the proof of \hyperlink{thm9982QN.100}{Theorem} \ref{thm9982QN.100}. The proof of the later appears in  \hyperlink{S9925Q}{Section} \ref{S9925Q}. Likewise, except for small changes the proof of \hyperlink{thm9982KR.100}{Theorem} \ref{thm9982KR.100} is the same as the proof of \hyperlink{thm9982KN.100}{Theorem} \ref{thm9982KN.100}. The proof of the later appears in  \hyperlink{S9925K}{Section} \ref{S9925K}. The last two sections cover other supporting results.

%SSSSSSSSSSSSSSSSSSSSSSSSSSSSSSSSSSSSSSSSSSSSSSSSSSSSSSSSSSSSSSSS
%SSSSSSSSSSSSSSSSSSSSSSSSSSSSSSSSSSSSSSSSSSSSSSSSSSSSSSSSSSSSSSSS
%SSSSSSSSSSSSSSSSSSSSSSSSSSSSSSSSSSSSSSSSSSSSSSSSSSSSSSSSSSSSSSSS
%SSSSSSSSSSSSSSSSSSSSSSSSSSSSSSSSSSSSSSSSSSSSSSSSSSSSSSSSSSSSSSSS
\section{Characteristic Functions of Quadratic  Residues}\label{S9911Q}\hypertarget{S9911Q}
For a prime $p$ the quadratic symbol modulo $p$ is defined by 
\begin{equation}\label{eq9911Q.100f}	\hypertarget{eq9911Q.100f}
	\left( \frac{a}{p}\right) 
	=\left \{
	\begin{array}{ll}
		1 & \text{ if } a \text{ is a quadratic residues},  \\
		-1 & \text{ if } a \text{ is a quadratic nonresidues}, \\
	\end{array} \right .\nonumber
\end{equation}
The classical characteristic functions of quadratic residues and quadratic nonresidues in the finite field $\F_p $ have the shapes
\begin{equation}\label{eq9911Q.100c-1}	\hypertarget{eq9911Q.100c-1}
	\varkappa _0(a)=\frac{1}{2}\left( 1+\left( \frac{a}{p}\right) \right) 
	=\left \{
	\begin{array}{ll}
		1 & \text{ if } a \text{ is a quadratic residues},  \\
		0 & \text{ if } a \text{ is a quadratic nonresidues}, \\
	\end{array} \right .\nonumber
\end{equation}
and 
\begin{equation}\label{eq9911Q.100c-2}	\hypertarget{eq9911Q.100c-2}
		\varkappa (a)=\frac{1}{2}\left( 1-\left( \frac{a}{p}\right) \right) 
		=\left \{
		\begin{array}{ll}
			1 & \text{ if } a \text{ is a quadratic nonresidues},  \\
			0 & \text{ if } a \text{ is a quadratic residues}, \\
		\end{array} \right .\nonumber
	\end{equation}
respectively. \\

The above indicator function was introduced over a century ago, an Euler criterion version of the indicator function is used in \cite{AN1896} to study consecutive pairs of quadratic residues and nonresidues and a modern version of the indicator function is used in \cite{JE1906} to study consecutive triples of quadratic residues and nonresidues. A new representation of the characteristic function for quadratic nonresidues in the finite field $\F_p $ is introduced here. 
%SSSSSSSSSSSSSSSSSSSSSSSSSSSSSSSSSSSSSSSSSSSSSSSSSSSSSSSSSSSSSSSS
%SSSSSSSSSSSSSSSSSSSSSSSSSSSSSSSSSSSSSSSSSSSSSSSSSSSSSSSSSSSSSSSS
%SSSSSSSSSSSSSSSSSSSSSSSSSSSSSSSSSSSSSSSSSSSSSSSSSSSSSSSSSSSSSSSS
%SSSSSSSSSSSSSSSSSSSSSSSSSSSSSSSSSSSSSSSSSSSSSSSSSSSSSSSSSSSSSSSS
\subsection{Characteristic Functions for Quadratic Residues}
New representations of the characteristic functions for quadratic residues and quadratic nonresidues in the finite field $\F_p $ are introduced here.  

\begin{lemma} \label{lem9911Q.200A} \hypertarget{lem9911Q.200A} Let \(p\geq 2\) be a prime and let \(\tau\) be a primitive root mod \(p\). If \(a\in
	\mathbb{F}_p\) is a nonzero element, then
\begin{enumerate}[font=\normalfont, label=(\roman*)]
\item$\displaystyle \varkappa_0 (a)=\sum _{0\leq n<p/2} \frac{1}{p}\sum _{0\leq s\leq p-1} e^ {i2\pi\frac{(\tau ^{2n}-a)s}{p}}
=\left \{
\begin{array}{ll}
	1 & \text{ if } a \text{ is a quadratic residues},  \\
	0 & \text{ if } a \text{ is a quadratic nonresidues}, \\
\end{array} \right .$ 
\item$\displaystyle \varkappa (a)=\sum _{0\leq n<p/2} \frac{1}{p}\sum _{0\leq s\leq p-1} e^ {i2\pi\frac{(\tau ^{2n+1}-a)s}{p}}
=\left \{
\begin{array}{ll}
	1 & \text{ if } a \text{ is a quadratic nonresidues},  \\
	0 & \text{ if } a \text{ is a quadratic residues}, \\
\end{array} \right .$ 
\end{enumerate}
\end{lemma}	
\begin{proof}[\textbf{Proof}] (ii) As the index \(n\geq 0\) ranges over the odd integers up to \(p/2\), the element \(\tau ^{2n+1}\in \mathbb{F}_p\) ranges over the quadratic nonresidues modulo $p$. Thus, the equation 
	\begin{equation}\label{eq9911Q.200c}
		\tau ^{2n+1}- a=0
	\end{equation} has a unique solution $n\in[1,p/2]$ if and only if the fixed element \(a\in \mathbb{F}_p\) is a quadratic nonresidue. In this case the inner sum in 
	\begin{equation}
		\sum _{0\leq n<p/2} \frac{1}{p}\sum _{0\leq s\leq p-1} e^ {i2\pi\frac{(\tau ^{2n+1}-a)s}{p}}
		=\left \{
		\begin{array}{ll}
			1 & \text{ if } a \text{ is a quadratic nonresidues},  \\
			0 & \text{ if } a \text{ is a quadratic residues}, \\
		\end{array} \right .
	\end{equation}	
	collapses to $\sum _{0\leq s\leq p-1}1=p$. This in turn identifies $a$ as a quadratic nonresidue. Otherwise, element \(a\in \mathbb{F}_p\) is not a quadratic nonresidue. Thus, the equation \eqref{eq9911Q.200c} has no solution and the inner sum collapses to $\sum _{0\leq s\leq p-1}e^ {i2\pi\frac{(\tau ^{2n+1}-a)s}{p}}=0$. This in turn identifies $a$ as a quadratic residue.
\end{proof}

%SSSSSSSSSSSSSSSSSSSSSSSSSSSSSSSSSSSSSSSSSSSSSSSSSSSSSSSSSSSSSSSS
%SSSSSSSSSSSSSSSSSSSSSSSSSSSSSSSSSSSSSSSSSSSSSSSSSSSSSSSSSSSSSSSS
%SSSSSSSSSSSSSSSSSSSSSSSSSSSSSSSSSSSSSSSSSSSSSSSSSSSSSSSSSSSSSSSS
%SSSSSSSSSSSSSSSSSSSSSSSSSSSSSSSSSSSSSSSSSSSSSSSSSSSSSSSSSSSSSSSS
\section{Characteristic Functions for $k$th Power Residues}
Let $p$ be a prime and let $k\geq2$ be an integer such that $k\mid p-1$. An element $a\in \F$ in a finite field is a $k$th power residue if and only if the congruence equation
\begin{equation} \label{eq9925K.200d}
	a^{(p-1)/k}\equiv1\bmod p 
\end{equation}	
is true. Otherwise the element is a $k$th power nonresidue.\\

New representations of the characteristic functions for $k$th power residues and $k$th power nonresidues in the finite field $\F_p $ are provided here.  

\begin{lemma} \label{lem9911K.200A} \hypertarget{lem9911K.200A} Let \(p\geq 2\) be a prime, let $k\mid p-1$, and let \(\tau\) be a primitive root mod \(p\). If \(a\in
	\mathbb{F}_p\) is a nonzero element, then
	\begin{enumerate}[font=\normalfont, label=(\roman*)]
		\item$\displaystyle \overline{\varkappa}_k (a)=\sum _{0\leq m<p/k} \frac{1}{p}\sum _{0\leq s\leq p-1} e^ {i2\pi\frac{(\tau ^{km}-a)s}{p}}
		=\left \{
		\begin{array}{ll}
			1 & \text{ if } a \text{ { is a $k$th power residues}},  \\
			0 & \text{ if } a \text{ is a $k$th power nonresidues}, \\
		\end{array} \right .$ 
		\item$\displaystyle \varkappa_k (a)=\sum _{0\leq m<p/k} \frac{1}{p}\sum _{0\leq s\leq p-1} e^ {i2\pi\frac{(\tau ^{km+1}-a)s}{p}}
		=\left \{
		\begin{array}{ll}
			1 & \text{ if } a \text{ is a $k$th power nonresidues},  \\
			0 & \text{ if } a \text{ is a $k$th power residues}, \\
		\end{array} \right .$ 
	\end{enumerate}
\end{lemma}	
\begin{proof}[\textbf{Proof}] (ii) As the index \(m\geq 0\) ranges over the integers up to \(p/k\), the element \(\tau ^{km+1}\in \mathbb{F}_p\) ranges over the $k$th power nonresidues modulo $p$. Thus, the equation 
	\begin{equation}
		\tau ^{km+1}- a=0
	\end{equation} has a unique solution $m\in[0,p/k]$ if and only if the fixed element \(a\in \mathbb{F}_p\) is a $k$th power nonresidue. In this case the inner sum in
\begin{equation}\label{eq9911K.200c}
	\sum _{0\leq m<p/k} \frac{1}{p}\sum _{0\leq s\leq p-1} e^ {i2\pi\frac{(\tau ^{km+1}-a)s}{p}}
	=\left \{
	\begin{array}{ll}
		1 & \text{ if } a \text{ is a $k$th power nonresidues},  \\
		0 & \text{ if } a \text{ is a $k$th power residues}, \\
	\end{array} \right .
\end{equation}	
collapses to $\sum _{0\leq s\leq p-1}1=p$. This in turn identifies $a$ as a $k$-power nonresidue. Otherwise, the element \(a\in \mathbb{F}_p\) is not a $k$th power nonresidue. Thus, the equation \eqref{eq9911K.200c} has no solution and the inner sum collapses to $\sum _{0\leq s\leq p-1}e^ {i2\pi\frac{(\tau ^{km+1}-a)s}{p}}=0$. Otherwise, element \(a\in \mathbb{F}_p\) is not a quadratic nonresidue. This in turn identifies $a$ as a $k$th power residue.
\end{proof} 
%SSSSSSSSSSSSSSSSSSSSSSSSSSSSSSSSSSSSSSSSSSSSSSSSSSSSSSSSSSSSSSSS
%SSSSSSSSSSSSSSSSSSSSSSSSSSSSSSSSSSSSSSSSSSSSSSSSSSSSSSSSSSSSSSSS
%SSSSSSSSSSSSSSSSSSSSSSSSSSSSSSSSSSSSSSSSSSSSSSSSSSSSSSSSSSSSSSSS
%SSSSSSSSSSSSSSSSSSSSSSSSSSSSSSSSSSSSSSSSSSSSSSSSSSSSSSSSSSSSSSSS
%SSSSSSSSSSSSSSSSSSSSSSSSSSSSSSSSSSSSSSSSSSSSSSSSSSSSSSSSSSSSSSSS
%SSSSSSSSSSSSSSSSSSSSSSSSSSSSSSSSSSSSSSSSSSSSSSSSSSSSSSSSSSSSSSSS
%SSSSSSSSSSSSSSSSSSSSSSSSSSSSSSSSSSSSSSSSSSSSSSSSSSSSSSSSSSSSSSSS
\section{Exponential Sums Results } \label{S9933P-ES}\hypertarget{S9933P-ES}
The exponential sums of interest in this analysis are presented in this section. 

An asymptotic relation for the exponential sums, which is a finite Fourier transform, occurring in the error term of the main result is provided here. 

\begin{thm}  \label{thm9933Q.110}\hypertarget{thm9933Q.110} {\normalfont (\cite{ML1972}) }  Let \(p\geq 2\) be a large prime, and let \(\tau \in \mathbb{F}_p\) be an element of large multiplicative order $\ord_p(\tau) \mid p-1$. Then, for any $b \in [1, p-1]$,  and $x\leq p-1$,
	\begin{equation}\label{eq9933Q.110d}
		\sum_{ n \leq x}  e^{i2\pi b \tau^{n}/p} \ll p^{1/2}  (\log p)^2.
	\end{equation}
\end{thm}

\begin{thm}   \label{thm9933ERP.220U}\hypertarget{thm9933ERP.220U}  Let \(p\geq 2\) be a large prime. If $\tau $ be a primitive root modulo $p$ and $a=o(p)$ is not a quadratic nonresidue, then
	\begin{equation} 
		\widehat{U(a)}=	\sum_{1\leq b\leq  p-1}	 e^{ \frac{-i2\pi ab}{p}}	\sum_{1\leq n<p/2} e^{\frac{i2\pi b \tau^{2n+1}}{p}} =-  \sum_{1\leq n< p/2} e^{\frac{i2\pi  \tau^{2n+1}}{p}} + O(p^{1/2} (\log p)^2)\nonumber,
	\end{equation} 
	where the implied constant is independent of $a, b \in [1, p-1]$. 	
\end{thm}  
\begin{proof}[\textbf{Proof}]The complete proof appears in {\color{red}\cite[Theorem 4.5]{CN2021}}. 
\end{proof}
%RRRRRRRRRRRRRRRRRRRRRRRRRRRRRRRRRRRRRRRRRRRRRRRRRRRRRRRRRRRRRRR

\begin{cor}   \label{cor9933ERP.230U}\hypertarget{cor9933ERP.230U}  Let \(p\geq 2\) be a large prime. If $\tau $ be a primitive root modulo $p$ and $0\ne a=o(p)$ is not a quadratic nonresidue, then
	\begin{equation} 
		\Bigg|\widehat{U(a)}\Bigg|=	\Bigg|\sum_{1\leq b\leq  p-1}	 e^{-i2\pi \frac{ab}{p}}	\sum_{1\leq n< p/2} e^{\frac{i2\pi b \tau^{2n+1}}{p}}\Bigg| = O(p^{1/2} (\log p)^2)\nonumber,
	\end{equation} 
	where the implied constant is independent of $a$ and $ b \in [1, p-1]$. 	
\end{cor} 
\begin{proof}[\textbf{Proof}] The second line in the estimation of the upper bound in \eqref{eq9933ERP.230d} follows from \hyperlink{thm9933ERP.220U}{Theorem} \ref{thm9933ERP.220U} and the fourth line follows from \hyperlink{thm9933Q.110}{Theorem} \ref{thm9933Q.110}: 
	\begin{eqnarray} \label{eq9933ERP.230d}
		\bigg|\widehat {U(a)}\bigg |&=& 	\Bigg|\sum_{1\leq b\leq  p-1}	 e^{\frac{-i2\pi ab}{p}}	\sum_{1\leq n< p/2} e^{\frac{i2\pi b \tau^{2n+1}}{p}}\Bigg| \\[.3cm]
		&=&\left |- \sum_{1\leq n< p/2} e^{ \frac{i2\pi\tau ^{2n+1}}{p}}+ O(p^{1/2} (\log p)^2 )  \right | \nonumber\\[.3cm]
		&\ll &\left |\sum_{1\leq n<p/2} e^{ \frac{i2\pi\tau ^{2n+1}}{p}}\right |+p^{1/2} (\log p)^2  \nonumber\\[.4cm]
		&\ll&  p^{1/2} (\log p)^2\nonumber,
	\end{eqnarray}
	where the implied constant is independent of $a\ne0$. 
\end{proof}

%SSSSSSSSSSSSSSSSSSSSSSSSSSSSSSSSSSSSSSSSSSSSSSSSSSSSSSSSSSSSSSSS
%SSSSSSSSSSSSSSSSSSSSSSSSSSSSSSSSSSSSSSSSSSSSSSSSSSSSSSSSSSSSSSSS
%SSSSSSSSSSSSSSSSSSSSSSSSSSSSSSSSSSSSSSSSSSSSSSSSSSSSSSSSSSSSSSSS
%SSSSSSSSSSSSSSSSSSSSSSSSSSSSSSSSSSSSSSSSSSSSSSSSSSSSSSSSSSSSSSSS
%SSSSSSSSSSSSSSSSSSSSSSSSSSSSSSSSSSSSSSSSSSSSSSSSSSSSSSSSSSSSSSSS
%SSSSSSSSSSSSSSSSSSSSSSSSSSSSSSSSSSSSSSSSSSSSSSSSSSSSSSSSSSSSSSSS
%SSSSSSSSSSSSSSSSSSSSSSSSSSSSSSSSSSSSSSSSSSSSSSSSSSSSSSSSSSSSSSSS
\section{Fibers and Multiplicities for Quadratic Residues} \label{S9925FMK}\hypertarget{S9925FMK}
The multiplicities of the fibers occurring in the estimate of the error term are computed in this section.
\begin{lemma}  \label{lem9900Q.300S}\hypertarget{lem9900Q.300S} Let $p$ be a prime, let $ x=o(p)$ and let $\tau\in \F_p$ be a primitive root in the finite field $\F_p$.  Define the maps
	\begin{equation}\label{eq9900Q.300-m}
		\alpha(s,n)\equiv (\tau ^{2s+1}-n)\bmod p\quad \text{ and } \quad 
		\beta(u,v)\equiv uv\bmod p.
	\end{equation}	
	Then, the fibers $\alpha^{-1}(m)$ and $\beta^{-1}(m)$ of an element $0\ne m\in \F_p$
	have the cardinalities 
	\begin{equation}\label{eq9900Q.300-f}
		\#	\alpha^{-1}(m)\leq x-1\quad \text{ and }\quad \#\beta^{-1}(m)=	x
	\end{equation}
	respectively.
\end{lemma}
\begin{proof}[\textbf{Proof}] Let $\mathscr{S}=\{s<p^{1-\varepsilon}\}$. Given a fixed $n\in [2,x]$, the map 
	\begin{equation}\label{eq9900Q.300-m1}
		\alpha:\mathscr{S}\times [2,x]\longrightarrow\F_p\quad  \text{ defined by }\quad  \alpha(s,n)\equiv (\tau ^{2s+1}-n)\bmod p,
	\end{equation}
	is one-to-one. This follows from the fact that the map $s\longrightarrow\tau^s \bmod p$ is a permutation the nonzero elements of the finite field $\F_p$, and the map $(s,n)\longrightarrow(\tau ^{2s+1}-n)\bmod p$ is a shifted permutation of the subset of quadratic nonresidues 
	\begin{equation}\label{eq9900Q.300-p}
		\mathscr{N}=\{\tau ^{2s+1}:s<p^{1-\varepsilon}\}\subset \F_p,
	\end{equation}
	see {\color{red}\cite[Chapter 7]{LN1997}} for more extensive details on the theory of permutation functions of finite fields. Thus, as $(s,n)\in \mathscr{S}\times [2,x] $ varies, each value $m=\alpha(s,n)$ is repeated at most $x-1$ times. Moreover, the premises no quadratic nonresidues $n\leq x=o(p)$ implies that $m=\alpha(s,n)\ne0$. This verifies that the cardinality of the fiber is
	\begin{eqnarray}\label{eq9900Q.300-f1}
		\#	\alpha^{-1}(m)&=&	\#\{(s,n)\in \mathscr{S}\times [2,x] :m\equiv (\tau ^{2s+1}-n)\bmod p\}\nonumber\\[.3cm]
		&\leq& x-1.
	\end{eqnarray}		
	Similarly, given a fixed $u\in [1,x]$, the map 
	\begin{equation}\label{eq9900Q.300-m2}
		\beta:[1,x]\times [1,p-1]\longrightarrow\F_p\quad  \text{ defined by }\quad  \beta(u,v)\equiv uv\bmod p,
	\end{equation}
	is one-to-one. Here the map $v\longrightarrow uv \bmod p$ permutes the elements of the finite field $\F_p$. Thus, as $(u,v)\in [1,x]\times [1,p-1]$ varies, each value $m=\beta(u,v)$ is repeated exactly $x$ times. This verifies that the cardinality of the fiber is
	\begin{eqnarray}\label{eq9900Q.300-f2}
		\#	\beta^{-1}(m)&=&	\#\{(u,v)\in [1,x]\times [1,p-1]:m\equiv uv\bmod p\}\nonumber\\[.2cm]
		&=&x.
	\end{eqnarray}
	
	Now each value $m=\alpha(s,n)\ne0$ (of multiplicity up to $(x-1)$ in $	\alpha^{-1}(m)$), is matched to $m=\alpha(s,n)=\beta(u,v)$ for some $(u,v)$, (of multiplicity exactly $x$ in $	\beta^{-1}(m)$). Comparing \eqref{eq9900Q.300-f1} and \eqref{eq9900Q.300-f2} prove that $\# \alpha^{-1}(m)\leq\# \beta^{-1}(m)$. 
\end{proof}

%%SSSSSSSSSSSSSSSSSSSSSSSSSSSSSSSSSSSSSSSSSSSSSSSSSSSSSSSSSSSSSSSS
%SSSSSSSSSSSSSSSSSSSSSSSSSSSSSSSSSSSSSSSSSSSSSSSSSSSSSSSSSSSSSSSS
%SSSSSSSSSSSSSSSSSSSSSSSSSSSSSSSSSSSSSSSSSSSSSSSSSSSSSSSSSSSSSSSS
%SSSSSSSSSSSSSSSSSSSSSSSSSSSSSSSSSSSSSSSSSSSSSSSSSSSSSSSSSSSSSSSS
\section{Evaluation of the Main Term}\label{S9925M}\hypertarget{S9925M}
The main term is evaluated in this section.
%TTTTTTTTTTTTTTTTTTTTTTTTTTTTTTTTTTTTTTTTTTTTTTTTTTTTTTTTTTTTTTTTTTTT
\begin{lemma} \label{lem9925M.300T}\hypertarget{lem9925M.300T} Let $\varepsilon>0$ be a small real number. Let $p\geq 2$ be a large prime, let $ x=o(p)$ and let $q=(\log\log p)$. If $1 \leq a <q$ is a pair of relatively prime integers, then 
	\begin{equation}
		\sum _{\substack{2 \leq n\leq x,\\n\equiv a\bmod q}}  \sum _{0\leq s<p/2} \frac{\Lambda(n)}{p}
		= \frac{x}{2\varphi(q)}+O\left( xe^{-c\sqrt{\log x}}\right),\nonumber
	\end{equation}
	where $c>0$ is a constant.
\end{lemma}

\begin{proof}[\textbf{Proof}] The number of relatively integers $n<p$ coincides with the values of the totient function. A routine rearrangement gives 
	\begin{eqnarray}\label{eq9955N.300f}
		\sum _{\substack{2 \leq n\leq x,\\n\equiv a\bmod q}}  \sum _{0\leq s<p/2} \frac{\Lambda(n)}{p} 
		&=&\frac{1}{p}\sum _{0\leq s<p/2,}\sum_{\substack{2 \leq n\leq x\\
				n \equiv a \bmod q		}} \Lambda(n)\\[.3cm]
		&=& \left( \frac{1}{2} +o(1)\right)  \left(  \frac{x}{\varphi(q)}+O\left( xe^{-c\sqrt{\log x}}\right)\right) \nonumber\\[.3cm]
		&=& \frac{x}{2\varphi(q)}+O\left( xe^{-c\sqrt{\log x}}\right),\nonumber
	\end{eqnarray}
	where $c>0$ is a constant.	Here the last line follows from  {\color{red}\cite[Corollary 5.29]{IK2004}}, {\color{red}\cite[Corollary 11.19]{MV2007}}, et cetera. 
\end{proof}

%SSSSSSSSSSSSSSSSSSSSSSSSSSSSSSSSSSSSSSSSSSSSSSSSSSSSSSSSSSSSSSSS
%SSSSSSSSSSSSSSSSSSSSSSSSSSSSSSSSSSSSSSSSSSSSSSSSSSSSSSSSSSSSSSSS
%SSSSSSSSSSSSSSSSSSSSSSSSSSSSSSSSSSSSSSSSSSSSSSSSSSSSSSSSSSSSSSSS
%SSSSSSSSSSSSSSSSSSSSSSSSSSSSSSSSSSSSSSSSSSSSSSSSSSSSSSSSSSSSSSSS
%SSSSSSSSSSSSSSSSSSSSSSSSSSSSSSSSSSSSSSSSSSSSSSSSSSSSSSSSSSSSSSSS
%SSSSSSSSSSSSSSSSSSSSSSSSSSSSSSSSSSSSSSSSSSSSSSSSSSSSSSSSSSSSSSSS
%SSSSSSSSSSSSSSSSSSSSSSSSSSSSSSSSSSSSSSSSSSSSSSSSSSSSSSSSSSSSSSSS
\section{Estimate For The Error Term} \label{S9925E}\hypertarget{S9925E}
A nontrivial upper bound for the error term $E_2(x,q,a)$ is computed in this section. The error term is partitioned as 
\begin{equation}
	E_2(x,q,a)=E_{0}(x,q,a)+E_{1}(x,q,a).
\end{equation} 
The upper bound of the first term $E_0(x,q,a)$ for $n< p/x$ is derived using a geometric summation/sine approximation techniques, and the upper bound of the second term $E_1(x,q,a)$ for $p/x\leq n\leq p/2$ is derived using exponential sums techniques.
%TTTTTTTTTTTTTTTTTTTTTTTTTTTTTTTTTTTTTTTTTTTTTTTTTTTTTTTT
\begin{lemma}  \label{lem9925E.300E}\hypertarget{lem9925E.300E} Let $p\geq 2$ be a large prime, let $ x\ll p^{1/2}(\log p)^{-2}$ and let $q=(\log\log p)$. If there is no quadratic nonresidue \(n\leq x\), then 
	\begin{equation}\label{eq9925E.300b}
		\sum _{\substack{2 \leq n\leq x,\\n\equiv a\bmod q}}  \sum _{0\leq s<p/2} \frac{\Lambda(n)}{p}\sum _{1\leq t\leq p-1} e^ {i2\pi\frac{(\tau ^{2s+1}-n)t}{p}} \ll (\log  p)(\log x)^2.\nonumber 
	\end{equation} 
\end{lemma}

\begin{proof}[\textbf{Proof}] The product of a point $(u,v)\in [1,x]\times [1,p/x)$ satisfies $uv<p$. This leads to the partition $[1,p/x)\cup[p/x,p/2)$ of the index $n$, which is suitable for the sine approximation $uv/p\ll\sin(\pi uv/p)\ll uv/p$ for $|uv/p|<1$ on the first subinterval $[1,p/x)$, see \eqref{eq9900P.300u1} for more details. Thus, consider the partition of the triple finite sum
	\begin{eqnarray} \label{eq9900P.300k}
		E_2(x,q,a)&=& \sum _{\substack{2 \leq n\leq x,\\n\equiv a\bmod q}}
		\frac{\Lambda(n)}{p}\sum_{s< p/2,} \sum_{ 1\leq t\leq p-1} e^{i2\pi \frac{(\tau ^{2s+1}-n)t}{p}}   \\
		&= & \sum _{\substack{2 \leq n\leq x,\\n\equiv a\bmod q}}
		\frac{\Lambda(n)}{p}\sum_{s< p/x,} \sum_{ 1\leq t\leq p-1} e^{i2\pi \frac{(\tau ^{2s+1}-n)t}{p}} \nonumber\\[.12cm]
		&&\hskip 1 in + \sum _{\substack{2 \leq n\leq x,\\n\equiv a\bmod q}}
		\frac{\Lambda(n)}{p}\sum_{p/x\leq s< p/2,} \sum_{ 1\leq t\leq p-1} e^{i2\pi \frac{(\tau ^{2s+1}-n)t}{p}} \nonumber\\[.2cm]
		&=&E_{0}(x,q,a)\;+\;E_{1}(x,q,a) \nonumber.
	\end{eqnarray} 
	Summing yields
	\begin{eqnarray} \label{eq9900P.300u4}
		E_2(x,q,a)&=& E_{0}(x,q,a)\;+\;E_{1}(x,q,a)   \\[.2cm]
		&\ll&  (\log x)^2(\log p)\;+\;  \frac{x(\log p)^2}{p^{1/2}}\nonumber\\[.12cm]
		&\ll& (\log x)^2(\log p)\nonumber,
		\end{eqnarray}
		for $x\ll p^{1/2}(\log p)^{-2}$.
	This completes the estimate of the error term.
\end{proof}

%LLLLLLLLLLLLLLLLLLLLLLLLLLLLLLLLLLLLLLLLLLLLLLLLLLLLLL
\begin{lemma}   \label{lem9900P.700}\hypertarget{lem9900P.700}  Let $p\geq 2$ be a large prime, let $ x=o(p)$ and let $q=(\log\log p)$. If $\tau $ be a primitive root modulo $p$ and there is no quadratic nonresidue \(n\leq x\), then
	\begin{equation} 
		E_{0}(x,q,a) = \sum _{\substack{2 \leq n\leq x,\\n\equiv a\bmod q}}
		\frac{\Lambda(n)}{p}\sum_{s< p/x,} \sum_{ 1\leq t\leq p-1} e^{i2\pi \frac{(\tau ^{2s+1}-n)t}{p}}= O((\log x)^2(\log p))\nonumber.
	\end{equation} 
\end{lemma} 
\begin{proof}[\textbf{Proof}]	To apply the geometric summation/sine function techniques, the subsum $E_0(x)$ is partition as follows.
	\begin{eqnarray} \label{eq9900P.300l}
		E_0(x,q,a)&=& \sum _{\substack{2 \leq n\leq x,\\n\equiv a\bmod q}}
		\frac{\Lambda(n)}{p}\sum_{s< p/x,} \sum_{ 1\leq t\leq p-1} e^{i2\pi \frac{(\tau ^{2s+1}-n)t}{p}}   \\
		&= & \sum _{\substack{2 \leq n\leq x,\\n\equiv a\bmod q}}
		\frac{\Lambda(n)}{p}\sum_{s< p/x} \left( \sum_{ 1\leq t< p/2} e^{i2\pi \frac{(\tau ^{2s+1}-n)t}{p}}+ \sum_{ p/2\leq t\leq p-1} e^{i2\pi \frac{(\tau ^{2s+1}-n)t}{p}}\right) \nonumber\\[.12cm]
		&=&E_{0,0}(x)\;+\;E_{0,1}(x) \nonumber.
	\end{eqnarray} 
	
	Now, a geometric series summation of the inner finite sum in the first term yields
	\begin{eqnarray} \label{eq9900P.300m}
		E_{0,0}(x)&=& \sum _{\substack{2 \leq n\leq x,\\n\equiv a\bmod q}}
		\frac{\Lambda(n)}{p}\sum_{s< p/x,}  \sum_{ 1\leq t< p/2} e^{i2\pi \frac{(\tau ^{2s+1}-n)t}{p}}  \\[.3cm]
		&=&   	 \sum _{\substack{2 \leq n\leq x,\\n\equiv a\bmod q}}\frac{\Lambda(n)}{p}\sum_{s<p/x}   \frac{e^{i2\pi (\frac{\tau ^{2s+1}-n}{p})(\frac{p}{2}+1)}-1}{1-e^{i2\pi \frac{(\tau ^{2s+1}-n)}{p}}} \nonumber\\[.3cm]
		&\leq&   	\frac{\log x}{p} \sum _{2 \leq n\leq x,}\sum_{s< p/x}   \frac{2}{|\sin\pi(\tau ^{2s+1}-n)/p|} \nonumber,
	\end{eqnarray} 
	see {\color{red}\cite[Chapter 23]{DH2000}} for similar geometric series calculation and estimation. The last line in \eqref{eq9900P.300m} follows from the hypothesis that $u$ is not a primitive root. Specifically, $0\ne\tau^{2s+1}-n\in \F_p$ for any $s < p/2$ and any $n\leq x$. Utilizing \hyperlink{lem9900Q.300S}{Lemma} \ref{lem9900Q.300S}, the first term has the upper bound
	\begin{eqnarray} \label{eq9900P.300u1}
		E_{0,0}(x)&=& \sum _{\substack{2 \leq n\leq x,\\n\equiv a\bmod q}}\frac{\Lambda(n)}{p}\sum_{s<p/x}   \frac{2}{|\sin\pi(\tau ^{2s+1}-n)/p|}\\	[.3cm]
		&\ll&  	\frac{2\log x}{p} \sum_{1\leq u\leq x,}\sum_{1\leq v< p/x}   \frac{1}{|\sin\pi uv/p|}\nonumber\\	[.3cm]
		&\ll&  	\frac{2\log x}{p} \sum_{1\leq u\leq x,}\sum_{1\leq v< p}   \frac{p}{\pi uv} \nonumber\\	[.3cm]
		&\ll& (\log x)	\sum_{1\leq u\leq x}\frac{1}{u}\sum_{1\leq v< p}   \frac{1}{v} \nonumber\\	[.3cm]
		&\ll& (\log x)^2(\log p)\nonumber,
	\end{eqnarray}
	where $uv<p$ and $|\sin\pi uv/p|\ne0$ since $p\nmid uv$. Similarly, the second term has the upper bound
	\begin{eqnarray} \label{eq9900P.300v}
		E_{0,1}(x)&=& \sum _{\substack{2 \leq n\leq x,\\n\equiv a\bmod q}} 
		\frac{\Lambda(n)}{p}\sum_{s<p/x,}  \sum_{ p/2\leq t\leq p-1} e^{i2\pi \frac{(\tau ^{2s+1}-n)t}{p}}  \\[.3cm]
		&=&   \sum _{\substack{2 \leq n\leq x,\\n\equiv a\bmod q}} 	\frac{\Lambda(n)}{p}\sum_{s<p/x}   \frac{1-e^{i2\pi (\frac{\tau ^s-n}{p})(\frac{p+1}{2})}}{1-e^{i2\pi \frac{(\tau ^s-n)}{p}}} \nonumber\\[.3cm]
		&\leq&   	\frac{\log x}{p} \sum_{2\leq n\leq x,}\sum_{s<p/x}  \frac{2}{|\sin\pi(\tau ^s-n)/p|} \nonumber\\[.3cm]
		&\ll&  (\log x)^2(\log p)\nonumber.
	\end{eqnarray}
	This is computed in the way as done in \eqref{eq9900P.300m} to \eqref{eq9900P.300u1}, mutatis mutandis. Hence, \\		
	\begin{equation}
		E_{0}(x)	=E_{0,0}(x)\;+\;E_{0,1}(x)\ll  (\log x)^2(\log p).
	\end{equation}
	
\end{proof}
%LLLLLLLLLLLLLLLLLLLLLLLLLLLLLLLLLLLLLLLLLLLLLLLLLLLLL
\begin{lemma}   \label{lem9900P.750}\hypertarget{lem9900P.750}  Let \(p\geq 2\) be a large primes. If $\tau $ be a primitive root modulo $p$ and there is no quadratic nonresidue \(n\leq x\), then
	\begin{equation} 
		E_{1}(x,q,a) = 	\sum _{\substack{2 \leq n\leq x,\\n\equiv a\bmod q}}
		\frac{\Lambda(n)}{p}\sum_{p/x\leq s< p/2,} \sum_{ 1\leq t\leq p-1} e^{i2\pi \frac{(\tau ^{2s+1}-n)t}{p}}= O\left( \frac{(\log p)^2}{p^{1/2}}	\cdot x\log x\right) \nonumber.
	\end{equation} 
\end{lemma} 
\begin{proof}[\textbf{Proof}] 	Rewrite the inner sum to obtain this:
	\begin{eqnarray} \label{eq9900P.750i}
		E_{1}(x,q,a)
		&=&  \sum _{\substack{2 \leq n\leq x,\\n\equiv a\bmod q}} 	\frac{\Lambda(n)}{p}\sum _{1\leq t\leq p-1}e^ {\frac{-i2\pi nt}{p}}\sum_{p/x\leq s< p/2}  e^ {i2 \pi \frac{\tau ^{2s+1}t}{p}}   \\[.2cm]
		&=&  \sum _{\substack{2 \leq n\leq x,\\n\equiv a\bmod q}}	\frac{\Lambda(n)}{p} \sum _{1\leq t\leq p-1}e^ {\frac{-i2\pi nt}{p}}\left( \sum_{1\leq s< p/2}  e^ {i2 \pi \frac{\tau ^{2s+1}}{p}}-\sum_{1\leq s<p/x}  e^ {i2 \pi \frac{\tau ^{2s+1}}{p}}\right)  \nonumber .
	\end{eqnarray}
	Now, use the estimate $\Lambda(n)\leq \log x$ for $n\leq x$, take absolute value and apply the triangle inequality:
	\begin{eqnarray} \label{eq9900P.750j}
		E_{1}(x,q,a)
		&\leq&   \frac{\log x}{p} \sum _{\substack{2 \leq n\leq x,\\n\equiv a\bmod q}}2\left|\sum _{1\leq t\leq p-1}e^ {\frac{-i2\pi nt}{p}}\sum_{1\leq s< p/2}   e^ {i2 \pi \frac{\tau ^{2s}}{p}}\right|\\[.3cm]
		&\ll&   \frac{\log x}{p} \sum _{2 \leq n\leq p/x} p^{1/2}(\log p)^2   \nonumber\\[.3cm]
		&\ll& \frac{(\log p)^2}{p^{1/2}}	\cdot x\log x\nonumber.
		\end{eqnarray}
The second line in \eqref{eq9900P.750j} follows from \hyperlink{cor9933ERP.230U}{Corollary} \ref{cor9933ERP.230U}. 
\end{proof}
		
 %SSSSSSSSSSSSSSSSSSSSSSSSSSSSSSSSSSSSSSSSSSSSSSSSSSSSSSSSSSSSSSSS
	%SSSSSSSSSSSSSSSSSSSSSSSSSSSSSSSSSSSSSSSSSSSSSSSSSSSSSSSSSSSSSSSS
	%SSSSSSSSSSSSSSSSSSSSSSSSSSSSSSSSSSSSSSSSSSSSSSSSSSSSSSSSSSSSSSSS
	%SSSSSSSSSSSSSSSSSSSSSSSSSSSSSSSSSSSSSSSSSSSSSSSSSSSSSSSSSSSSSSSS
	\section{Small Prime Quadratic Nonresidues in Arithmetic Progression}\label{S9925Q}\hypertarget{S9925Q}
	Define the counting function
	\begin{equation} \label{eq9925Q.400d}
		N_2(x,q,a)	=\#\left\{r\leq x: \left( \frac{r}{p}\right)=-1 \text{ and }r\equiv a\bmod q \text{ is prime}\right\}.
	\end{equation}	 
	The determination of an upper bound for the smallest quadratic nonresidue in arithmetic modulo $p$ is based on a new characteristic function for quadratic nonresidue in finite field $\F_p$ developed in \hyperlink{S9911Q}{Section} \ref{S9911Q} and the weighted characteristic function of prime numbers, (better known as the vonMangoldt function),
	\begin{equation} \label{eq9955P.400i}
		\Lambda(n)=
		\begin{cases}
			\log p&\text{ if } n=p^k,\\
			0&\text{ if } n\ne p^k,
		\end{cases}
	\end{equation}
	where $p^k$ is a prime power. The counting function \eqref{eq9925Q.400d} is easily derived from the weighted counting function \eqref{eq9925Q.400h}. 
	
	\begin{proof}[{\color{blue}\normalfont\textbf{Proof of \hyperlink{thm9982QN.100}{Theorem} \ref{thm9982QN.100}}}] Let \(p\) be a large prime number, let $x=(\log p)(\log\log p)^{3+\varepsilon}$, where $\varepsilon>0$ is a small number. Suppose the least quadratic nonresidue $n>x$ and consider the sum of the characteristic function over the short interval \([2,x]\), that is, 
		\begin{equation} \label{eq9925Q.400h}
			N_2(x,q,a)	=\sum _{\substack{2 \leq n\leq x\\n\equiv a\bmod q}}\varkappa(n) \Lambda(n)=0.
		\end{equation}
		Replacing the characteristic function for quadratic nonresidues, \hyperlink{lem9911Q.200A}{Lemma} \ref{lem9911Q.200A}, and expanding the nonexistence equation \eqref{eq9925Q.400h} yield
		\begin{eqnarray} \label{eq9925Q.400m}
			N_2(x,q,a)&=&\sum _{\substack{2 \leq n\leq x\\n\equiv a\bmod q}} \varkappa (n)\Lambda(n)  \\
			&=&\sum _{\substack{2 \leq n\leq x\\n\equiv a\bmod q}} \Lambda(n) \sum _{0\leq s<p/2} \frac{1}{p}\sum _{0\leq t\leq p-1} e^ {i2\pi\frac{(\tau ^{2s+1}-n)t}{p}} \nonumber\\[.3cm] 
			&=&\sum _{\substack{2 \leq n\leq x,\\n\equiv a\bmod q}}  \sum _{0\leq s<p/2} \frac{\Lambda(n)}{p}  +\sum _{\substack{2 \leq n\leq x,\\n\equiv a\bmod q}}  \sum _{0\leq s<p/2} \frac{\Lambda(n)}{p}\sum _{1\leq t\leq p-1} e^ {i2\pi\frac{(\tau ^{2s+1}-n)t}{p}}\nonumber\\[.3cm] 
			&=&M_2(x,q,a)\; +\; E_2(x,q,a)\nonumber.
		\end{eqnarray} 
		
		The main term $M_2(x,q,a)$ is determined by $s=0$, which reduces the exponential to \(e^{i 2\pi  ks/p}=1\), it is evaluated in \hyperlink{lem9925M.300T}{Lemma} \ref{lem9925M.300T}. The error term $E_2(x,q,a)$ is determined by $s\ne0$, which retains  the exponential to \(e^{i 2\pi  ks/p}\ne1\), it is estimated in \hyperlink{lem9925E.300E}{Lemma} \ref{lem9925E.300E}. \\
		
		Substituting these evaluation, estimate and value $x=(\log p)(\log\log p)^{3+\varepsilon}$ yield
		\begin{eqnarray} \label{eq9925Q.400p}
			N_2(x,q,a)&=&\sum _{\substack{2 \leq n\leq x\\n\equiv a\bmod q}} \varkappa (n)\Lambda(n)		\\[.3cm]	
			&=&M(x,q,a) + E(x,q,a) \nonumber\\[.3cm]
			&=&\left[ \frac{x}{2\varphi(q)}+O\left( xe^{-c\sqrt{\log x}}\right)\right] +\left[ O\left((\log p)(\log x)^2\right) \right] \nonumber\\[.3cm]
			&=&\left[ \frac{(\log p)(\log\log p)^{3+\varepsilon}}{2\varphi(q)}+O\left( (\log p)(\log\log p)^{3+\varepsilon}e^{-c\sqrt{\log\log p}}\right)\right] \nonumber\\[.3cm]
			&&\hskip 2.5 in +\left[ O\left((\log p)(\log \log p)^2\right) \right] \nonumber\\[.3cm]
			&=& \frac{(\log p)(\log\log p)^{3+\varepsilon}}{2\varphi(q)}+O\left((\log p)(\log \log p)^2\right)  \nonumber,
		\end{eqnarray} 
		where $c>0$ is a constant. For constraint $q\ll\log\log p$, the totient function $\varphi(q)$ is bounded away from zero, that is,
		\begin{equation}\label{eq9925Q.400s}
			0<\frac{1}{\log \log p}\ll\frac{1}{\varphi(q)},
		\end{equation}	
		see {\color{red}\cite[Theorem 2.9]{MV2007}}. Consequently, for small parameter $q=O(\log \log p)$ the main term in \eqref{eq9925Q.400p} dominates the error term:
		
		\begin{eqnarray} \label{eq9925Q.400v}
			N_2(x,q,a)&=&\sum _{\substack{2 \leq n\leq x\\n\equiv a\bmod q}} \varkappa (n)\Lambda(n)		\\[.3cm]	
			&=&\frac{(\log p)(\log\log p)^{3+\varepsilon}}{2\varphi(q)}+O\left((\log p)(\log \log p)^2\right)\nonumber\\[.3cm]
			&\gg&  (\log p)(\log\log p)^{2+\varepsilon}\left( 1+O\left(\frac{1}{(\log\log p)^{\varepsilon}}\right) \right) \nonumber\\[.3cm]
			&>&0\nonumber 
		\end{eqnarray} 
		as $p\to\infty$. Clearly, this contradicts the hypothesis \eqref{eq9925Q.400h} for all sufficiently large prime numbers $p \geq p_0$. Therefore, there exists a prime quadratic nonresidue in the arithmetic progression 
		\begin{equation}
			n=qm+a\leq (\log p)(\log\log p)^{3+\varepsilon},	
		\end{equation}
		modulo $p$.
	\end{proof}

%SSSSSSSSSSSSSSSSSSSSSSSSSSSSSSSSSSSSSSSSSSSSSSSSSSSSSSSSSSSSSSSS
%SSSSSSSSSSSSSSSSSSSSSSSSSSSSSSSSSSSSSSSSSSSSSSSSSSSSSSSSSSSSSSSS
%SSSSSSSSSSSSSSSSSSSSSSSSSSSSSSSSSSSSSSSSSSSSSSSSSSSSSSSSSSSSSSSS
%SSSSSSSSSSSSSSSSSSSSSSSSSSSSSSSSSSSSSSSSSSSSSSSSSSSSSSSSSSSSSSSS
\section{Numerical Data}\label{EXA9955L}\hypertarget{EXA9955L}
The data for a small prime are computed here to illustrate the concept.

%EEEEEEEEEEEEEEEEEEEEEEEEEEEEEEEEEEEEEEEEEEEEEEEEEEEEEEEEEEE
\subsection{Numerical Data for Prime Quadratic Residues}
\begin{exa}\label{exa9925K.100A}{\normalfont For the prime closest to $10^{24}$, the parameters are these: 
		\begin{enumerate}[font=\normalfont, label=(\alph*)]
			
			\item $\displaystyle p=10^{24}+7,$ \tabto{8cm}prime,\\
			
			\item $\displaystyle x=(\log p)(\log\log p)^{3}=3568.93,$\tabto{8cm}range with $\varepsilon=0$,\\
			
			\item $\displaystyle q\leq\log\log p=4.0,$\tabto{8cm}modulo,\\
			
			\item $\displaystyle R_2\gg\frac{(\log p)(\log\log p)^{2}}{2\varphi(q)}=892.23,   $\tabto{8cm}predicted number $N_2$ in \eqref{eq9925Q.400v}.
			
		\end{enumerate}
The estimate in $R=R(x,q,a)$ above is a weighted count and the unweighted number of prime quadratic residues is $R_0=R_0(x,q,a)$, see the definitions of the counting functions in \eqref{eq9925K.400d} and \eqref{eq9925K.400h}. Thus, the predicted number of small prime quadratic residues is
\begin{equation}\label{eq9925K.400dh}
	R_0(x,q,a)=\delta_r(q,a)\cdot  \frac{1}{\varphi(q)}\cdot (\log p)(\log\log p)^{2}+Error \approx  222.39+Error,
\end{equation} 
where $\delta_r(q,a)=c_r(q,a)/2$ is the density of prime quadratic residues and $c_r(q,a)\geq0$ is a correction factor, depending on $q=4$ and $a=1,3$. The first 10 small prime quadratic residues in each arithmetic progression modulo $q=4$ are these:
		\begin{enumerate}%[font=\normalfont, label=(\roman*)]
			\item $\begin{aligned}[t] \mathscr{R}_1&=\{n=4m+1\leq x:\left( \frac{n}{p}\right)=1\}\\&=\{29,61,73,89,109,137,151,181,197,271\ldots\},\end{aligned}$
			
			\item $\begin{aligned}[t] \mathscr{R}_3&=\{n=4m+3\leq x:\left( \frac{n}{p}\right)=1\}\\&=\{3,19,23,43,47,71,79,83,87,103,\ldots\}.\end{aligned}$
		\end{enumerate}

	}
\end{exa}

%EEEEEEEEEEEEEEEEEEEEEEEEEEEEEEEEEEEEEEEEEEEEEEEEEEEEEEEEEEE
\subsection{Numerical Data for Prime Quadratic Nonresidues}
\begin{exa}\label{exa9925N.200A}{\normalfont For the prime closest to $10^{24}$, the parameters are these: 
		\begin{enumerate}[font=\normalfont, label=(\alph*)]
	
	\item $\displaystyle p=10^{24}+7,$ \tabto{8cm}prime,\\
	
	\item $\displaystyle x=(\log p)(\log\log p)^{3}=3568.93,$\tabto{8cm}range with $\varepsilon=0$,\\
	
	\item $\displaystyle q\leq\log\log p=4.0,$\tabto{8cm}modulo,\\
	
	\item $\displaystyle N_2\gg\frac{(\log p)(\log\log p)^{2}}{2\varphi(q)}=892.23,   $\tabto{8cm}predicted number $N_2$ in \eqref{eq9925Q.400v}.
	
\end{enumerate}
		\vskip .1 in
The estimate in $N_2=N_2(x,q,a)$ above is a weighted count and the unweighted number of prime quadratic nonresidues is $\overline{N_2}=\overline{N_2}(x,q,a)$, see the definitions of the counting functions in \eqref{eq9925Q.400d} and \eqref{eq9925Q.400h}. Thus, the predicted number of small quadratic nonresidues is 
\begin{equation}\label{eq9925Q.400di}
\overline{	N_2}(x,q,a)=\frac{N_2(x,q,a)}{\log \log p}=  \frac{c(q,a)}{2\varphi(q)}\cdot (\log p)(\log\log p)^{2}+Error,
\end{equation}
where $c(q,a)\geq0$ is a correction factor, depending on $q=4$ and $a=1,3$.  The first 10 small prime quadratic nonresidues in each arithmetic progression modulo $q=4$ are these:
		\begin{enumerate}%[font=\normalfont, label=(\roman*)]
			\item $\begin{aligned}[t]
				\mathscr{N}_1&=\{n=4m+1\leq x:\left( \frac{n}{p}\right)=-1\}\\&=\{5,13,17,37,53,97,101,113,173,229,\ldots\},
			\end{aligned}$
			
			\item $\begin{aligned}[t] \mathscr{N}_3&=\{n=4m+3\leq x:\left( \frac{n}{p}\right)=-1\}\\&=\{59,67,107,139,167,179,211,223,227,239,\ldots\}.	\end{aligned}$
		\end{enumerate}	
	}
\end{exa}

%SSSSSSSSSSSSSSSSSSSSSSSSSSSSSSSSSSSSSSSSSSSSSSSSSSSSSSSSSSSSSSSS
%SSSSSSSSSSSSSSSSSSSSSSSSSSSSSSSSSSSSSSSSSSSSSSSSSSSSSSSSSSSSSSSS
%SSSSSSSSSSSSSSSSSSSSSSSSSSSSSSSSSSSSSSSSSSSSSSSSSSSSSSSSSSSSSSSS
\section{Small Prime $k$th Power Nonresidues in Arithmetic Progressions}\label{S9925K}\hypertarget{S9925K}
Define the counting function
\begin{equation} \label{eq9925K.400d}
	N_k(x,q,a)	=\#\left\{n\leq x: n^{(p-1)/k}\equiv-1\bmod p \text{ and }n\equiv a\bmod q \text{ is prime}\right\}.
\end{equation}	 
The determination of an upper bound for the smallest $k$th power nonresidue  modulo $p$ in arithmetic modulo $q$ is based on a new characteristic function for quadratic nonresidue in finite field $\F_p$ developed in \hyperlink{S9911Q}{Section} \ref{S9911Q} and the weighted characteristic function of prime numbers, (better known as the vonMagoldt function),
\begin{equation} \label{eq9955K.400i}
	\Lambda(n)=
	\begin{cases}
		\log p&\text{ if } n=p^m,\\
		0&\text{ if } n\ne p^m,
	\end{cases}
\end{equation}
where $p^m$ is a prime power. The counting function \eqref{eq9925K.400d} is easily derived from the weighted counting function \eqref{eq9925K.400h}. 

\begin{proof}[{\color{blue}\normalfont\textbf{Proof of \hyperlink{thm9982KR.100}{Theorem} \ref{thm9982KR.100}}}] Let $p$ be a large prime number, let $x=(\log p)(\log\log p)^{4+\varepsilon}$, where $\varepsilon>0$ is a small number and let $k<x$ be a small integer such that $k\mid p-1$. Suppose the least $k$th power nonresidue $n>x$ and consider the sum of the characteristic function over the short interval $[2,x]$, that is, 
	\begin{equation} \label{eq9925K.400h}
		\overline{N_p}(x,q,a)=\sum _{\substack{2 \leq n\leq x\\n\equiv a\bmod q}}\varkappa_k(n) \Lambda(n)=0.
	\end{equation}
	Replacing the characteristic function for quadratic nonresidues, \hyperlink{lem9911K.200A}{Lemma} \ref{lem9911K.200A}, and expanding the nonexistence equation \eqref{eq9925K.400h} yield
	\begin{eqnarray} \label{eq9925K.400m}
		\overline{N_k}(x,q,a)&=&\sum _{\substack{2 \leq n\leq x\\n\equiv a\bmod q}} \varkappa_k (n)\Lambda(n)  \\
		&=&\sum _{\substack{2 \leq n\leq x\\n\equiv a\bmod q}} \Lambda(n) \sum _{0\leq s<p/k} \frac{1}{p}\sum _{0\leq t\leq p-1} e^ {i2\pi\frac{(\tau ^{ks+1}-n)t}{p}} \nonumber\\[.3cm] 
		&=&\sum _{\substack{2 \leq n\leq x,\\n\equiv a\bmod q}}  \sum _{0\leq s<p/k} \frac{\Lambda(n)}{p}  +\sum _{\substack{2 \leq n\leq x,\\n\equiv a\bmod q}}  \sum _{0\leq s<p/k} \frac{\Lambda(n)}{p}\sum _{1\leq t\leq p-1} e^ {i2\pi\frac{(\tau ^{ks+1}-n)t}{p}}\nonumber\\[.3cm] 
		&=&M_k(x,q,a)\; +\; E_k(x,q,a)\nonumber.
	\end{eqnarray} 
	
	The main term $M_k(x,q,a)$ is determined by $t=0$, which reduces the exponential to \(e^{i 2\pi  ks/p}=1\), it is evaluated in \hyperlink{lem9925KM.300M}{Lemma} \ref{lem9925KM.300M}. The error term $E_k(x,q,a)$ is determined by $t\ne0$, which retains  the exponential to \(e^{i 2\pi  ks/p}\ne1\), it is estimated in \hyperlink{lem9925KE.300E}{Lemma} \ref{lem9925KE.300E}. \\
	
Substituting these evaluation, estimate and the value $x$ yield
	\begin{eqnarray} \label{eq9925K.400p}
		\overline{N_k}(x,q,a)&=&\sum _{\substack{2 \leq n\leq x\\n\equiv a\bmod q}} \varkappa_k (n)\Lambda(n)		\\[.3cm]	
		&=&M_k(x,q,a) + E_k(x,q,a) \nonumber\\[.3cm]
		&=&\left[ \frac{1}{k}\cdot \frac{x}{\varphi(q)}+O\left( xe^{-c\sqrt{\log x}}\right)\right] +\left[ O\left((\log p)(\log x)^2\right) \right] \nonumber\\[.3cm]
		&=&\left[\frac{1}{k}\cdot \frac{(\log p)(\log\log p)^{4+\varepsilon}}{\varphi(q)}+O\left( (\log p)(\log\log p)^{4+\varepsilon}\cdot e^{-c_1\sqrt{\log\log p}}\right)\right] \nonumber\\[.3cm]
		&&\hskip 2.78 in +\left[ O\left((\log p)(\log \log p)^2\right) \right] \nonumber,
	\end{eqnarray} 
	where $c, c_1>0$ are constants. For any integer $q\ll\log\log p$, the ratio is bounded away from zero, that is,
	\begin{equation}\label{eq9925K.400s}
		\frac{1}{\varphi(q)}\gg\frac{1}{\log \log p}>0 \nonumber,
	\end{equation}	
since $\varphi(q)\gg q/\log \log q$. Consequently, for small parameters $k\ll\log \log p$ and $q\ll\log \log p$ the main term in \eqref{eq9925K.400p} dominates the error term:
	
	\begin{eqnarray} \label{eq9925K.400v}
		N_k(x,q,a)&=&\sum _{\substack{2 \leq n\leq x\\n\equiv a\bmod q}} \varkappa_k (n)\Lambda(n)		\\[.3cm]	
		&=&\frac{1}{k}\cdot \frac{(\log p)(\log\log p)^{4+\varepsilon}}{\varphi(q)}+O\left((\log p)(\log \log p)^2\right)\nonumber\\[.3cm]
		&\gg& (\log p)(\log\log p)^{2+\varepsilon}+O\left((\log p)(\log \log p)^2\right) \nonumber\\[.3cm]
		&\gg& (\log p)(\log\log p)^{2+\varepsilon}\left( 1+O\left(\frac{1}{(\log\log p)^{\varepsilon}}\right) \right) \nonumber\\[.3cm]
		&>&0\nonumber
	\end{eqnarray} 
as $p\to\infty$. Clearly, this contradicts the hypothesis \eqref{eq9925K.400h} for all sufficiently large prime numbers $p \geq p_0$. Therefore, there exists a prime $k$th power nonresidue in the arithmetic progression
\begin{equation}
	 n=qm+a\leq (\log p)(\log\log p)^{4+\varepsilon}
\end{equation}
as claimed.
\end{proof}

%%SSSSSSSSSSSSSSSSSSSSSSSSSSSSSSSSSSSSSSSSSSSSSSSSSSSSSSSSSSSSSSSS
%SSSSSSSSSSSSSSSSSSSSSSSSSSSSSSSSSSSSSSSSSSSSSSSSSSSSSSSSSSSSSSSS
%SSSSSSSSSSSSSSSSSSSSSSSSSSSSSSSSSSSSSSSSSSSSSSSSSSSSSSSSSSSSSSSS
%SSSSSSSSSSSSSSSSSSSSSSSSSSSSSSSSSSSSSSSSSSSSSSSSSSSSSSSSSSSSSSSS
\subsection{Evaluation of the Main Term for $k$th Powers}\label{S9925KM}\hypertarget{S9925KM}
The main term is evaluated in this section.
%TTTTTTTTTTTTTTTTTTTTTTTTTTTTTTTTTTTTTTTTTTTTTTTTTTTTTTTTTTTTTTTTTTTT
\begin{lem} \label{lem9925KM.300M}\hypertarget{lem9925KM.300M} Let $\varepsilon>0$ be a small real number. Let \(p\geq 2\) be a large prime, let $ x=o( p)$ and let $k, q\ll (\log\log p)$ be a pair of integers. If $1 \leq a <q$ is a pair of relatively prime integers and $k\mid p-1$ then 
	\begin{equation}
		\sum _{\substack{2 \leq n\leq x,\\n\equiv a\bmod q}}  \sum _{0\leq m<p/k} \frac{\Lambda(n)}{p}
		=\frac{1}{k}\cdot \left( \frac{x}{\varphi(q)}+O\left( xe^{-c\sqrt{\log x}}\right)\right) ,\nonumber
	\end{equation}
	where $c>0$ is a constant.
\end{lem}

\begin{proof}[\textbf{Proof}] A routine rearrangement gives 
	\begin{eqnarray}\label{eq9955KM.300f}
		\sum _{\substack{2 \leq n\leq x,\\n\equiv a\bmod q}}  \sum _{0\leq m<p/k} \frac{\Lambda(n)}{p}
		&=&\frac{1}{p}\sum _{0\leq m<p/k,}\sum_{\substack{2 \leq n\leq x\\
				n \equiv a \bmod q		}} \Lambda(n)\\[.3cm]
		&=& \left( \frac{1}{k} +o(1)\right)  \left(   \frac{x}{\varphi(q)}+O\left( xe^{-c\sqrt{\log x}}\right)\right) \nonumber\\[.3cm]
		&=& \frac{1}{k}\cdot \left( \frac{x}{\varphi(q)}+O\left( xe^{-c\sqrt{\log x}}\right)\right) ,\nonumber
	\end{eqnarray}
	where $c>0$ is a constant.	Here the last second follows from the prime number theorem over arithmetic progressions, {\color{red}\cite[Corollary 5.29.]{IK2004}}, {\color{red}\cite[Corollary 11.19]{MV2007}}, et cetera. 
\end{proof}

%SSSSSSSSSSSSSSSSSSSSSSSSSSSSSSSSSSSSSSSSSSSSSSSSSSSSSSSSSSSSSSSS
%SSSSSSSSSSSSSSSSSSSSSSSSSSSSSSSSSSSSSSSSSSSSSSSSSSSSSSSSSSSSSSSS
%SSSSSSSSSSSSSSSSSSSSSSSSSSSSSSSSSSSSSSSSSSSSSSSSSSSSSSSSSSSSSSSS
%SSSSSSSSSSSSSSSSSSSSSSSSSSSSSSSSSSSSSSSSSSSSSSSSSSSSSSSSSSSSSSSS
\subsection{Estimate For The Error Term for $k$th Powers} \label{S9925KE}\hypertarget{S9925KE}
A nontrivial upper bound for the error term $E_k(x,q,a)$ is computed in this section. The error term is partitioned as 
\begin{equation}
	E_k(x,q,a)=E_{k,0}(x,q,a)+E_{k,1}(x,q,a).
\end{equation} 
The upper bound of the first term $E_{k,0}(x,q,a)$ for $s< p/x$ is derived using a geometric summation/sine approximation techniques, and the upper bound of the second term $E_{k,1}(x,q,a)$ for $p/x\leq s\leq p/k$ is derived using exponential sums techniques.
%TTTTTTTTTTTTTTTTTTTTTTTTTTTTTTTTTTTTTTTTTTTTTTTTTTTTTTTT
\begin{lem}  \label{lem9925KE.300E}\hypertarget{lem9925KE.300E} Let $p\geq 2$ be a large prime, let $ x\ll p^{1/2}(\log p)^{-2}$ and let $q=(\log\log p)$. If $k<x$ and there is no quadratic nonresidue \(n\leq x\), then 
	\begin{equation}\label{eq9925KP.300b}
		\sum _{\substack{2 \leq n\leq x,\\n\equiv a\bmod q}}  \sum _{0\leq s<p/k} \frac{\Lambda(n)}{p}\sum _{1\leq t\leq p-1} e^ {i2\pi\frac{(\tau ^{ks+1}-n)t}{p}} \ll (\log  p)(\log x)^2.\nonumber 
	\end{equation} 
\end{lem}

\begin{proof}[\textbf{Proof}] The product of a point $(u,v)\in [1,x]\times [1,p/x)$ satisfies $uv<p$. This leads to the partition $[1,p/x)\cup[p/x,p/k)$ of the index $n$, which is suitable for the sine approximation $uv/p\ll\sin(\pi uv/p)\ll uv/p$ for $|uv/p|<1$ on the first subinterval $[1,p/x)$. Thus, consider the partition of the triple finite sum
	\begin{eqnarray} \label{eq9900KP.300k}
		E_k(x,q,a)&=& \sum _{\substack{2 \leq n\leq x,\\n\equiv a\bmod q}}
		\frac{\Lambda(n)}{p}\sum_{s< p/k,} \sum_{ 1\leq t\leq p-1} e^{i2\pi \frac{(\tau ^{2s+1}-n)t}{p}}   \\[.2cm]
		&= & \sum _{\substack{2 \leq n\leq x,\\n\equiv a\bmod q}}
		\frac{\Lambda(n)}{p}\sum_{s< p/x,} \sum_{ 1\leq t\leq p-1} e^{i2\pi \frac{(\tau ^{2s+1}-n)t}{p}} \nonumber\\[.2cm]
		&&\hskip 1.25 in+ \sum _{\substack{2 \leq n\leq x,\\n\equiv a\bmod q}}
		\frac{\Lambda(n)}{p}\sum_{p/x\leq s< p/k,} \sum_{ 1\leq t\leq p-1} e^{i2\pi \frac{(\tau ^{2s+1}-n)t}{p}} \nonumber\\[.3cm]
		&=&E_{k,0}(x,q,a)\;+\;E_{k,1}(x,q,a) \nonumber.
	\end{eqnarray} 
	Summing yields
	\begin{eqnarray} \label{eq9900KP.300u4}
		E_k(x,q,a)&=& E_{k,0}(x,q,a)\;+\;E_{k,1}(x,q,a)   \\
		&\ll&  (\log x)^2(\log p)\;+\;  \frac{x(\log p)^2}{p^{1/2}}\nonumber\\[.12cm]
		&\ll& (\log x)^2(\log p)\nonumber,
	\end{eqnarray}
for $x\ll p^{1/2}(\log p)^{-2}$.	This completes the estimate of the error term.
\end{proof}

%TTTTTTTTTTTTTTTTTTTTTTTTTTTTTTTTTTTTTTTTTTTTTTTTTTTTTTTT
\begin{lem}  \label{lem9925KE.300E1}\hypertarget{lem9925KE.300E1} Suppose \(p\) is a large prime and $x=o(p)$. If there is no $k$th power nonresidue $n\leq x$, then 
	\begin{equation}\label{eq9925KE.300b}
		\left |\sum _{\substack{2 \leq n\leq x,\\n\equiv a\bmod q}}  \sum _{0\leq s<p/k} \frac{\Lambda(n)}{p}\sum _{1\leq t\leq p-1} e^ {i2\pi\frac{(\tau ^{ks+1}-n)t}{p}} \right |\ll (\log  p)(\log x)^2.\nonumber 
	\end{equation} 
\end{lem}
\begin{proof}[\textbf{Proof}] To compute an upper bound consider the symmetric partition of the triple finite sum
	\begin{eqnarray} \label{eq9925KE.300k}
		E_{k,0}(x,q,a)&=& \sum _{\substack{2 \leq n\leq x,\\n\equiv a\bmod q}}  \sum _{0\leq s<p/x} \frac{\Lambda(n)}{p}\sum _{1\leq t\leq p-1} e^ {i2\pi\frac{(\tau ^{ks+1}-n)t}{p}}  \\
		&= & \sum _{\substack{2 \leq n\leq x,\\n\equiv a\bmod q}}  \sum _{0\leq s<p/x} \frac{\Lambda(n)}{p} \left( \sum_{ 0<t\leq p/2} e^ {i2\pi\frac{(\tau ^{ks+1}-n)t}{p}}+ \sum_{ p/2<t\leq p-1} e^ {i2\pi\frac{(\tau ^{ks+1}-n)t}{p}}\right) \nonumber\\[.12cm]
		&=&T_{1}(x,q,a)\;+\;T_{2}(x,q,a) \nonumber.
	\end{eqnarray} 
	A geometric series summation of the inner finite sum in the first term yields
	\begin{eqnarray} \label{eq9925KE.300m}
		T{1}(x,q,a)&=& \sum _{\substack{2 \leq n\leq x,\\n\equiv a\bmod q}}  \sum _{0\leq s<p/x} \frac{\Lambda(n)}{p} \sum_{ 0<t\leq p/2} e^ {i2\pi\frac{(\tau ^{ks+1}-n)t}{p}} \\[.3cm]
		&=&   	\frac{1}{p} \sum _{\substack{2 \leq n\leq x,\\n\equiv a\bmod q}}  \sum _{0\leq s<p/x} \Lambda(n) \frac{e^{i2\pi (\frac{\tau ^{ks+1}-n}{p})(\frac{p+1}{2})}-1}{e^{i2\pi \frac{(\tau ^{ks+1}-n)}{p}}-1} \nonumber\\[.3cm]
		&\leq&   	\frac{\log x}{p} \sum _{2 \leq n\leq x,}  \sum _{0\leq s<p/x}   \frac{2}{|\sin\pi(\tau ^{ks+1}-n)/p|} \nonumber,
	\end{eqnarray} 
	see {\color{red}\cite[Chapter 23]{DH2000}} for similar geometric series calculation and estimation. The last line in \eqref{eq9925KE.300m} follows from the hypothesis that $n$ is not a $k$th power nonresidue. Specifically, $0\ne\tau^{km+1}-n\in \F_p$ for any $m <p/k$ and any $n\leq x$. \\
	
	Now apply \hyperlink{lem9900KE.300K}{Lemma} \ref{lem9900KE.300K} to obtain the next inequality
	\begin{eqnarray} \label{eq9925KE.300t}
		T_{1}(x,q,a)&\leq&   \frac{\log x}{p} \sum _{2 \leq n\leq x,}  \sum _{0\leq s<p/x}   \frac{2}{|\sin\pi(\tau ^{ks+1}-n)/p|} \\	[.3cm]
		&\leq&  	\frac{2\log x}{p}  \sum _{1 \leq u\leq x,}  \sum _{1\leq v<p}   \frac{1}{|\sin\pi uv/p|}\nonumber ,
	\end{eqnarray} 
where $|\sin\pi uv/p|\ne0$ since $p\nmid uv$. 	Observe that the index in the first double sum in \eqref{eq9925KE.300t} ranges over $(x-1)$ copies of a set of cardinality $<p/k$, for example $\mathscr{A}=\{1,2,3,\ldots, (p-1)/x\}$. In contrast, the index in the last double sum in \eqref{eq9925KE.300t} ranges over $x$ copies of the set $[1,p-1]$. Thus, \hyperlink{lem9900KE.300K}{Lemma} \ref{lem9900KE.300K} fully justifies the inequality \eqref{eq9925KE.300t}. In light of this information, continuing the calculation of the estimate as in {\color{red}\cite[p.\;136]{DH2000}}, yields 
	\begin{eqnarray} \label{eq9925KE.300u}
		T_{1}(x,q,a)
		&\ll&  	\frac{2\log x}{p} \sum_{1\leq u\leq x,}\sum_{1\leq v< p}   \frac{p}{\pi uv} \\	[.3cm]
		&\ll& (\log x)	\sum_{1\leq u\leq x}\frac{1}{u}\sum_{1\leq v< p}   \frac{1}{v} \nonumber\\	[.3cm]
		&\ll& (\log x)^2(\log p)\nonumber,
	\end{eqnarray}
	where $|\sin\pi uv/p|\ne0$ since $p\nmid uv$. Similarly, the second term has the upper bound
	\begin{eqnarray} \label{eq9925KE.300v}
		T_{2}(x,q,a)&=& \sum _{\substack{2 \leq n\leq x,\\n\equiv a\bmod q}}  \sum _{0\leq s<p/x} \frac{\Lambda(n)}{p} \sum_{ p/2<s\leq p} e^ {i2\pi\frac{(\tau ^{km+1}-n)s}{p}} \\[.3cm]
		&=&    \frac{1}{p}\sum _{\substack{2 \leq n\leq x\\n\equiv a\bmod q}} \Lambda(n) \sum _{0\leq s<p/x} \frac{1-e^{i2\pi (\frac{\tau ^{ks+1}-n}{p})(\frac{p+1}{2})}}{e^{i2\pi \frac{(\tau ^{ks+1}-n)}{p}}-1} \nonumber\\[.3cm]
		&\leq&   	\frac{(\log x)}{p} \sum_{2\leq n\leq x,} \sum _{0\leq s<p/x}  \frac{2}{|\sin\pi(\tau ^{ks+1}-n)/p|} \nonumber\\[.3cm]
		&\ll&  (\log x)^2(\log p)\nonumber.
	\end{eqnarray}
	Adding \eqref{eq9925KE.300v} and \eqref{eq9925KE.300v} yield
	\begin{eqnarray} \label{eq9900KE.300x}
		E_{k,0}(x,q,a)&=& T_{1}(x,q,a)\;+\;T_{2}(x,q,a,)  \\[.3cm]
		&\ll&   (\log p)(\log  x)^2+(\log p)(\log x)^2 \nonumber\\[.3cm]
		&\ll& (\log p)(\log x)^2\nonumber.
	\end{eqnarray}
	This completes the verification.
\end{proof}

%LLLLLLLLLLLLLLLLLLLLLLLLLLLLLLLLLLLLLLLLLLLLLLLLLLLLL
\begin{lemma}   \label{lem9900KP.750}\hypertarget{lem9900KP.750}  Let \(p\geq 2\) be a large primes, let $x=o(p)$ and let $k<x$. If $\tau $ be a primitive root modulo $p$ and there is no quadratic nonresidue \(n\leq x\), then
	\begin{equation} 
		E_{k,1}(x,q,a) = 	\sum _{\substack{2 \leq n\leq x,\\n\equiv a\bmod q}}
		\frac{\Lambda(n)}{p}\sum_{p/x\leq s< p/k,} \sum_{ 1\leq t\leq p-1} e^{ \frac{(i2\pi\tau ^{ks+1}-n)t}{p}}= O\left( \frac{(\log p)^2}{p^{1/2}}	\cdot x\log x\right) \nonumber.
	\end{equation} 
\end{lemma} 
\begin{proof}[\textbf{Proof}] 	Rewrite the inner sum to obtain this:
	\begin{eqnarray} \label{eq9900KP.750i}
		E_{k,1}(x,q,a)
		&=&  \sum _{\substack{2 \leq n\leq x,\\n\equiv a\bmod q}} 	\frac{\Lambda(n)}{p}\sum _{1\leq t\leq p-1}e^ {\frac{-i2\pi nt}{p}}\sum_{p/x\leq s< p/k}  e^ { \frac{i2 \pi\tau ^{ks+1}t}{p}}   \\[.2cm]
		&=&  \sum _{\substack{2 \leq n\leq x,\\n\equiv a\bmod q}}	\frac{\Lambda(n)}{p} \sum _{1\leq t\leq p-1}e^ {\frac{-i2\pi nt}{p}}\left( \sum_{1\leq s< p/k}  e^ { \frac{i2 \pi\tau ^{ks+1}}{p}}-\sum_{1\leq s<p/x}  e^ { \frac{i2 \pi\tau ^{ks+1}}{p}}\right)  \nonumber .
	\end{eqnarray}
	Now, use the estimate $\Lambda(n)\leq \log x$ for $p\leq x$, take absolute value and apply the triangle inequality:
	\begin{eqnarray} \label{eq9900KP.750j}
		E_{k,1}(x,q,a)
		&\leq&   \frac{\log x}{p} \sum _{\substack{2 \leq n\leq x,\\n\equiv a\bmod q}}2\left|\sum _{1\leq t\leq p-1}e^ {\frac{-i2\pi nt}{p}}\sum_{1\leq s< p/2}   e^ {i2 \pi \frac{\tau ^{2s}}{p}}\right|\\[.3cm]
		&\ll&   \frac{\log x}{p} \sum _{2 \leq n\leq p/x} p^{1/2}(\log p)^2   \nonumber\\[.3cm]
		&\ll& \frac{(\log p)^2}{p^{1/2}}	\cdot x\log x\nonumber.
	\end{eqnarray}
	The second line in \eqref{eq9900KP.750j} follows from \hyperlink{cor9933ERP.230U}{Corollary} \ref{cor9933ERP.230U}. 
\end{proof}

%SSSSSSSSSSSSSSSSSSSSSSSSSSSSSSSSSSSSSSSSSSSSSSSSSSSSSSSSSSSSSSSS
%SSSSSSSSSSSSSSSSSSSSSSSSSSSSSSSSSSSSSSSSSSSSSSSSSSSSSSSSSSSSSSSS
%SSSSSSSSSSSSSSSSSSSSSSSSSSSSSSSSSSSSSSSSSSSSSSSSSSSSSSSSSSSSSSSS
%SSSSSSSSSSSSSSSSSSSSSSSSSSSSSSSSSSSSSSSSSSSSSSSSSSSSSSSSSSSSSSSS
\subsection{Permutations and Fibers} %\label{S9925PF}\hypertarget{S9925PF}

%TTTTTTTTTTTTTTTTTTTTTTTTTTTTTTTTTTTTTTTTTTTTTTTTTTTTTTTT
\begin{lem}  \label{lem9900KE.300K}\hypertarget{lem9900KE.300K} Let $p\geq 2$ be a prime, let $k\mid p-1$, let $ x=o(p)$ and let $\tau\in \F_p$ be a primitive root in the finite field $\F_p$.  Define the maps
	\begin{equation}\label{eq9900KE.300-m}
		\alpha(m,n)\equiv (\tau ^{km+1}-n)\bmod p\quad \text{ and } \quad 
		\beta(u,v)\equiv uv\bmod p.
	\end{equation}	
	Then, the fibers $\alpha^{-1}(r)$ and $\beta^{-1}(r)$ of an element $0\ne r\in \F_p$
	have the cardinalities 
	\begin{equation}\label{eq9900KE.300-f}
		\#	\alpha^{-1}(r)\leq x-1\quad \text{ and }\quad \#\beta^{-1}(r)=	x
	\end{equation}
	respectively.
\end{lem}
\begin{proof}[\textbf{Proof}]Let $\mathscr{M}=\{m<p/k\}$. Given a fixed $n\in [2,x]$, the map 
	\begin{equation}\label{eq9900KE.300-m1}						  		
		\alpha:\mathscr{M}\times [2,x] \longrightarrow\F_p\quad  \text{ defined by }\quad  \alpha(m,n)\equiv (\tau ^{km+1}-n)\bmod p,
	\end{equation}
	is one-to-one. This follows from the fact that the map $s\longrightarrow\tau^s \bmod p$ is a permutation the nonzero elements of the finite field $\F_p$, and the restriction map $m\longrightarrow(\tau ^{km+1}-n)\bmod p$ is a shifted permutation of the subset of $k$th power nonresidues 
	\begin{equation}\label{eq9900KE.300-p}
		\mathscr{N}_k=\{\tau ^{km+1}:m\geq0\}\subset \F_p,
	\end{equation}
	see {\color{red}\cite[Chapter 7]{LN1997}} for more details on the theory of permutations of finite fields. Thus, as $(m,n)\in \mathscr{M}\times [2,x]$ varies, each value $r=\alpha(m,n)$ is repeated at most $x-1$ times. Moreover, the premises no $k$th power nonresidues $n\leq x=(\log p)^{1+\varepsilon}$ implies that $r=\alpha(m,n)\ne0$. This verifies that the cardinality of the fiber
	\begin{eqnarray}\label{eq9900KE.300-f1}
		\#	\alpha^{-1}(r)&=&	\#\{(m,n):r\equiv (\tau ^{km+1}-n)\bmod p:2\leq n\leq x \text{ and } m<p/k\}\nonumber\\
		&\leq &x-1.
	\end{eqnarray}		
	Similarly, given a fixed $u\in [1,x]$, the map 
	\begin{equation}\label{eq9900KE.300-m2}
		\beta:[1,x]\times [1,p-1]\longrightarrow\F_p\quad  \text{ defined by }\quad  \beta(u,v)\equiv uv\bmod p,
	\end{equation}
	is one-to-one. Here the map $v\longrightarrow uv \bmod p$ permutes the nonzero elements of the finite field $\F_p$. Thus, as $(u,v)\in [1,x]\times[1,p-1]$ varies, each value $m=\beta(u,v)$ is repeated exactly $x$ times. This verifies that the cardinality of the fiber
	\begin{equation}\label{eq9900KE.300-f2}
		\#	\beta^{-1}(r)=	\#\{(u,v):r\equiv uv\bmod p:1\leq u\leq x \text{ and }1\leq v< p\}=x
	\end{equation}
	
	Now each value $r=\alpha(m,n)\ne0$ (of multiplicity at most $(x-1)$ in $	\alpha^{-1}(r)$), is matched to $r=\alpha(m,n)=\beta(u,v)$ for some $(u,v)$, (of multiplicity exactly $x$ in $	\beta^{-1}(r)$). Comparing \eqref{eq9900KE.300-f1} and \eqref{eq9900KE.300-f2} proves that $\# \alpha^{-1}(m)\leq\# \beta^{-1}(m)$.
\end{proof}

%SSSSSSSSSSSSSSSSSSSSSSSSSSSSSSSSSSSSSSSSSSSSSSSSSSSSSSSSSSSSSSSS
%SSSSSSSSSSSSSSSSSSSSSSSSSSSSSSSSSSSSSSSSSSSSSSSSSSSSSSSSSSSSSSSS
%SSSSSSSSSSSSSSSSSSSSSSSSSSSSSSSSSSSSSSSSSSSSSSSSSSSSSSSSSSSSSSSS
%SSSSSSSSSSSSSSSSSSSSSSSSSSSSSSSSSSSSSSSSSSSSSSSSSSSSSSSSSSSSSSSS
\section{Numerical Data for $k$th Power Nonresidues}\label{EXA9955K}\hypertarget{EXA9955K}
The data for a small prime are computed here to illustrate the concept.

%EEEEEEEEEEEEEEEEEEEEEEEEEEEEEEEEEEEEEEEEEEEEEEEEEEEEEEEEEEE
\subsection{Numerical Data for Prime $3$th Power Nonresidues}
The $3$th power nonresidue symbol is defined by
\begin{equation}
	\left( \frac{n}{p}\right)_3\equiv n^{(p-1)/3}\not\equiv 1\bmod p.
\end{equation}
Some numerical estimates for the cubic nonresidues are made below. 
\begin{exa}\label{exa9925KN.100B}{\normalfont For the prime closest to $2^{128}$, the parameters are these: 
		\begin{enumerate}[font=\normalfont, label=(\alph*)]
			
			\item $\displaystyle p=2^{128}+51,$ \tabto{8cm}a random 39-digit prime,\\
			\item $\displaystyle k=3\mid p-1,$ \tabto{8cm}third power nonresidues,\\
			\item $\displaystyle x=(\log p)(\log p)^{4}=35915.80,$\tabto{8cm}range with $\varepsilon=0$,\\
			
			\item $\displaystyle q\leq(\log\log p)^2=20.12,$\tabto{8cm}range of moduli,\\
			
			\item $\displaystyle \overline{N_3}\gg\frac{(\log p)(\log p)^{4}}{k\varphi(q)}=2992.98,$\tabto{8cm}predicted number $\overline{N_3}$ in \eqref{eq9925K.400v}.
			
		\end{enumerate}
		The estimate in $\overline{N_3}=\overline{N_3}(x,q,a)$ above is a weighted count and the unweighted number of prime quadratic residues is $N_3=N_3(x,q,a)$, see the definitions of the counting functions in \eqref{eq9925K.400d} and \eqref{eq9925K.400h}. Thus, the predicted number of small quadratic residues in each residue class $1,3,5,7 \bmod 8$ is
		\begin{eqnarray}\label{eq9925K.400dh-3}
			N_3(x,q,a)&=&c(k,q,a)\frac{1}{k}\cdot  \frac{1}{\varphi(q)}\cdot \frac{(\log p)^{1+\varepsilon}}{ \log (\log p)^{1+\varepsilon}}+Error\nonumber\\
			&=& 667c(k,q,a)+Error,
		\end{eqnarray} 
		where $c(k,q,a)\geq0$ is a correction factor, depending on $k=3$, $q=8$ and $a=1,3,5,7$. The actual number of small $3$th power nonresidues in each arithmetic progression modulo $q=8$ are these, both composite and primes are listed for $n\leq x=1000$:
		\begin{enumerate}%[font=\normalfont, label=(\roman*)]
			\item $\begin{aligned}[t] \mathscr{N}_1&=\{n=8m+1\leq x:\left( \frac{n}{p}\right)_3\ne1\}\\&=\{161,{\color{red}193},{\color{red}233},{\color{red}241},{\color{red}409},{\color{red}617},825,993\ldots\},\end{aligned}$
			
			\item $\begin{aligned}[t] \mathscr{N}_3&=\{n=8m+3\leq x:\left( \frac{n}{p}\right)_3\ne1\}\\&=\{27,35,{\color{red}59},{\color{red}163},171,203,507,{\color{red}563},635,{\color{red}643},{\color{red}659},675,\ldots\}.\end{aligned}$
			
			\item $\begin{aligned}[t] \mathscr{N}_5&=\{n=8m+5\leq x:\left( \frac{n}{p}\right)_3\ne1\}\\&=\{{\color{red}29},117,245,{\color{red}269},413,517,669,741,{\color{red}821},{\color{red}877},941,989,\ldots\},\end{aligned}$
			
			\item $\begin{aligned}[t] \mathscr{N}_7&=\{n=8m+7\leq x:\left( \frac{n}{p}\right)_3\ne1\}\\&=\{{\color{red}23},215,303,335,423,{\color{red}463},591,671,711,{\color{red}727},775,879,{\color{red}919},999,\ldots\}.\end{aligned}$
			
		\end{enumerate}

		The $3$th power nonresidues are elements of multiplicative order $$\ord_p u=(p-1)/3=113427455640312821154458202477256070502$$	in $\F_p$, where 
		$p=2^{128}+51$.
	}
\end{exa}

%EEEEEEEEEEEEEEEEEEEEEEEEEEEEEEEEEEEEEEEEEEEEEEEEEEEEEEEEEEE
\subsection{Numerical Data for Prime $7$th Power Nonresidues}
The $7$th power nonresidue symbol is defined by
\begin{equation}
	\left( \frac{n}{p}\right)_7\equiv n^{(p-1)/7}\not\equiv 1\bmod p.
\end{equation}
\begin{exa}\label{exa9925KN.100A}{\normalfont For the prime closest to $10^{48}+200$, the parameters are these: 
		\begin{enumerate}[font=\normalfont, label=(\alph*)]
			
			\item $\displaystyle p=10^{48}+217,$ \tabto{8cm}a 48-digit random prime,\\
			\item $\displaystyle k=7\mid p-1,$ \tabto{8cm}seventh power nonresidues,\\
			\item $\displaystyle x=(\log p)(\log p)^{4}=54172.84,$\tabto{8cm}range with $\varepsilon=0$,\\
			
			\item $\displaystyle q\leq(\log\log p)^2=22.14,$\tabto{8cm}range of moduli,\\
			
			\item $\displaystyle \overline{N_7}\gg\frac{(\log p)(\log p)^{4}}{k\varphi(q)}=3869.49,$\tabto{8cm}predicted number $\overline{N_7}$ in \eqref{eq9925K.400v}.
			
		\end{enumerate}
		The estimate in $\overline{N_7}=\overline{N_7}(x,q,a)$ above is a weighted count and the unweighted number of prime quadratic residues is $N_7=N_7(x,q,a)$, see the definitions of the counting functions in \eqref{eq9925K.400d} and \eqref{eq9925K.400h}. Thus, the predicted number of small quadratic residues is
		\begin{eqnarray}\label{eq9925K.400dh-7}
			N_7(x,q,a)&=&c(k,q,a)\frac{1}{k}\cdot  \frac{1}{\varphi(q)}\cdot \frac{(\log p)(\log \log p)^{4}}{ \log (\log p)(\log \log p)^{4}}+Error\nonumber\\
			&=& 822c(k,q,a)+Error,
		\end{eqnarray} 
where $c(k,q,a)\geq0$ is a correction factor, depending on $k=7$, $q=3$ and $a=1,2$. The actual number of small $7$th power nonresidues in each arithmetic progression modulo $q=3$ are these, both composite and primes $n\leq x=300$ are listed here:
		\begin{enumerate}%[font=\normalfont, label=(\roman*)]
			\item $\begin{aligned}[t] \mathscr{N}_1&=\{n=3m+1\leq x:\left( \frac{n}{p}\right)_7\ne1\}\\&=\{{\color{red}19},{\color{red}61},136,{\color{red}139},247,{\color{red}283},322,\ldots\},\end{aligned}$
			
			\item $\begin{aligned}[t] \mathscr{N}_3&=\{n=3m+2\leq x:\left( \frac{n}{p}\right)_7\ne1\}\\&=\{35,80,{\color{red}83},158,{\color{red}173},209,{\color{red}281},\ldots\}.\end{aligned}$
		\end{enumerate}	
		
The $7$th power nonresidues are elements of multiplicative order
$$\ord_p u=(p-1)/7=142857142857142857142857142857142857142857142888$$	in $\F_p$, where 
$p=10^{48}+217.$

	}
\end{exa}

%SSSSSSSSSSSSSSSSSSSSSSSSSSSSSSSSSSSSSSSSSSSSSSSSSSSSSSSSSSSSSSSS
%SSSSSSSSSSSSSSSSSSSSSSSSSSSSSSSSSSSSSSSSSSSSSSSSSSSSSSSSSSSSSSSS
%SSSSSSSSSSSSSSSSSSSSSSSSSSSSSSSSSSSSSSSSSSSSSSSSSSSSSSSSSSSSSSSS
%SSSSSSSSSSSSSSSSSSSSSSSSSSSSSSSSSSSSSSSSSSSSSSSSSSSSSSSSSSSSSSSS
%\setcounter{secnumdepth}{1} % levels under \section are not numbered
%\setcounter{tocdepth}{2}    % levels under \subsection are not listed in the TOC
\section{Research Problems}\label{S8895RP}%\hypertarget{S8895RP}
%\setcounter{subsection}{1}
%PPPPPPPPPPPPPPPPPPPPPPPPPPPPPPPPPPPPPPPPPPPPPPPPPPPPPPPPPPPPP
\subsection{Correction Constants for $k$th Power Residues and Nonresidues}%\label{S8895P}\hypertarget{S8895P}
The correction factor arises as a correction to the density, it accounts for the statistical dependency among certain variables. The earliest observation of this phenomenon seems to the discrepancy observed in the primes counting function for primes with a fixed primitive roots, a historical account of its development appears in \cite{SP2003}. This is topic is known as \textit{entanglement}, some materials on the advanced theory of entanglement groups appears in \cite{LM2014}, \cite{PT2022}, et alii. Likely, the densities for primes $k$th power residues and nonresidues are given by exclusion-inclusion infinite sums such as
\begin{equation}\label{S8895RP.300k}
\delta(k,q,a)\stackrel{?}{=}	\sum_{n\geq1}\frac{\mu(n)}{[\mathcal{K}_n:\Q]},
\end{equation}   
where $\mathcal{K}_n$ is a sequence of field extensions of the rational numbers $\Q$. \\

A research problem of interest is to determine the correction $c(k,q,a)$ to the density $\delta(k,q,a)=c(k,q,a)/k$ associated with the prime $k$th power nonresidues counting function
\begin{align}
N_k(x,q,a)&=c(k,q,a)\cdot \frac{1}{k}\cdot  \frac{1}{\varphi(q)}\cdot \frac{(\log p)(\log\log p)^{4+\varepsilon}}{ \log (\log p)(\log\log p)^{4+\varepsilon}}\\[.4cm] 
& \hskip  2in +O\left(    \log p)(\log\log p)^{4+\varepsilon}  e^{-c\sqrt{\log\log p} }\right), \nonumber
\end{align}
where $c>0$ is a constant. The correction factor $c(k,q,a)\geq0$ fine tunes the density $\delta(k,q,a)$ to account for the properties of being a $k$th power nonresidue, being a prime, and being in a specific arithmetic progression.

%PPPPPPPPPPPPPPPPPPPPPPPPPPPPPPPPPPPPPPPPPPPPPPPPPPPPPPPPPPPPP
\subsection{New Patterns of Consecutive Quadratic Residues and Nonresidues}\label{S8895P}\hypertarget{S8895P}
The smallest quadratic nonresidue modulo $p$ is a prime, but not all small quadratic nonresidues modulo $p$ are primes. For example, in $\F_{41}^{\times}$, the subset of quadratic residues and quadratic nonresidues are these respectively.

\begin{enumerate}%[font=\normalfont, label=(\roman*)]
	\item $\displaystyle \mathscr{Q}=\{1, 2,4, 5,8,9, 10,16, 18,20,21,23,25,31,32,33, 36,   37,39,40\}$ and \\
	
	\item $\displaystyle \mathscr{N}=\{3,6,7,11,12,13,14,15,17,19,22,24,26,27,28,29,30,34,35,38\}.$
\end{enumerate}	

Some of the question on the binary patterns $RR$, $RN$, $NR$ and $NN$ of consecutive quadratic residues $R$ and nonresidues $N$ have been resolved completely, see {\color{red}\cite[Chapter 6]{BW1998}}. But the restricted patterns of consecutive prime/composite quadratic residues and nonresidues, such as $R_pR_p$, $R_pR_c $, $R_cR_p$, $R_cR_c$, has no literature. These require the evaluations of more complicated finite sums. \\

%PPPPPPPPPPPPPPPPPPPPPPPPPPPPPPPPPPPPPPPPPPPPPPPPPPPPPP
%PPPPPPPPPPPPPPPPPPPPPPPPPPPPPPPPPPPPPPPPPPPPPPPPPPPPPP
%PPPPPPPPPPPPPPPPPPPPPPPPPPPPPPPPPPPPPPPPPPPPPPPPPPPPPP
%PPPPPPPPPPPPPPPPPPPPPPPPPPPPPPPPPPPPPPPPPPPPPPPPPPPPPP
\addtocontents{toc}{\setcounter{tocdepth}{1}} 
\section{Problems}\label{EXE5577I}

%TTTTTTTTTTTTTTTTTTTTTTTTTTTTTTTTTTTTTTTTT

\begin{exe} \label{exe8895.101} {\normalfont Determine the number of consecutive pairs of prime quadratic nonresidues, composite quadratic nonresidues such as $(13,14), (29,30)$ in $\F_{41}$ but in a large finite field $\F_p$, the finite sum has the form
		\begin{equation} \label{S8895P.200f}
			\frac{1}{4}\sum _{2 \leq n\leq x}\left( 1-\left(\frac{n}{p} \right)\right) \left( 1-\left(\frac{n+1}{p} \right)\right)   \Lambda(n)=\sum _{2 \leq n\leq x}\varkappa(n)\varkappa(n+1) \Lambda(n).
		\end{equation}
		Either finite sum is suitable for this problem.
	}
\end{exe}
\vskip .15 in
%EEEEEEEEEEEEEEEEEEEEEEEEEEEEEEEEEEEEEEEEEEEEEEEEEEEEE

\begin{exe} \label{exe8895.104} {\normalfont Assume the twin prime conjecture. Determine the density of consecutive prime quadratic nonresidues with respect to the set of twin primes in large finite field $\F_p$. For example evaluate or estimate the counting function 
		\begin{equation} \label{S8895P.200g}
			\frac{1}{4}\sum _{2 \leq n\leq x}\left( 1-\left(\frac{n}{p} \right)\right) \left( 1-\left(\frac{n+2}{p} \right)\right)   \Lambda(n)\Lambda(n+2).
		\end{equation}
		In the finite field $\F_{41}$ the density is 1/2 since there are 2 pairs $11,13$ and $(17,19)$ of twin primes and consecutive quadratic nonresidue among the total number of twin prime pairs $(3,5),(5,7),(11,13),(17,19)$ in the interval $[1,41]$.\\
		
		The same statistics should be true for consecutive quadratic residues.
	}
\end{exe}
\vskip .15 in
%EEEEEEEEEEEEEEEEEEEEEEEEEEEEEEEEEEEEEEEEEEEEEEEEEEEEE

\begin{exe} \label{exe8895.107} {\normalfont Prove or disprove whether a pair of any pattern of consecutive quadratic residues are not equidistributed on the interval $[1,p-1]$. Likewise, any pattern of consecutive nonquadratic residues are not equidistributed on the interval $[1,p-1]$.
	}
\end{exe}
\vskip .15 in
%EEEEEEEEEEEEEEEEEEEEEEEEEEEEEEEEEEEEEEEEEEEEEEEEEEEEE
\begin{exe} \label{exe8895.101} {\normalfont Define the gap between two pairs of consecutive residues $r,r+1$ and $r+d,r+d+1$ by the difference $gap=r+d-(r+1)=d-1$, where $d\geq2$. Estimate the average gap between consecutive quadratic residues over the interval $[1,p-1]$. Likewise, estimate the average gap between consecutive quadratic nonresidues over the interval $[1,p-1]$.
	}
\end{exe}
\vskip .15 in
%EEEEEEEEEEEEEEEEEEEEEEEEEEEEEEEEEEEEEEEEEEEEEEEEEEEEE

%BBBBBBBBBBBBBBBBBBBBBBBBBBBBBBBBBBBBBBBBBBBBBBBBBBBBBBBBBBBBBBBB
%BBBBBBBBBBBBBBBBBBBBBBBBBBBBBBBBBBBBBBBBBBBBBBBBBBBBBBBBBBBBBBBB
%BBBBBBBBBBBBBBBBBBBBBBBBBBBBBBBBBBBBBBBBBBBBBBBBBBBBBBBBBBBBBBBB
%BBBBBBBBBBBBBBBBBBBBBBBBBBBBBBBBBBBBBBBBBBBBBBBBBBBBBBBBBBBBBBBB
%%%%%%%%%%%%%%%%%%%%%%%%%%%%%% Bibliography%%%%%%%%%%%%%%%%%%%%%%
%\newpage

%\currfilename.\\


\begin{thebibliography}{999}
%AAAAAAAAAAAAAAAAAAAAAAAAAAAAAAAA

%\bibitem{AN1952}Ankeny, N. C. \textit{\color{red}The least quadratic non residue.} Ann. of Math. (2) 55 (1952), 65-72. \href{https://mathscinet.ams.org/mathscinet-getitem?mr=0045159}{MR0045159}.

%\bibitem{AP1976} Apostol, T. M. \textit{\color{red}Introduction to analytic number theory}. Undergraduate Texts in Mathematics. Springer-Verlag, New York-Heidelberg, 1976. \href{https://mathscinet.ams.org/mathscinet-getitem?mr=0434929}{MR0434929}.

\bibitem{AN1896} Aladov, N.S. \textit{\color{red}Sur la distribution des residus quadratiques et non-quadratiques d?un nombre
	premier $p$ dans la suite $1,2,...,p-1$.} Mat. Sb., 18 (1896), no. 1, 61--75.

%BBBBBBBBBBBBBBBBBBBBBBBBBBBBBBBBBBBBBBBBBBBBBBBB


\bibitem{BW1998} Berndt, B.; Evans, R.; Williams, K. \textit{\color{red}Gauss and Jacobi sums.} Canad. Math. Soc. Ser. Monogr. John Wiley \& Sons, Inc., New York, 1998.
\href{https://mathscinet.ams.org/mathscinet-getitem?mr=MR1625181}{MR1625181}.

%\bibitem{BD1957} Burgess, D. A. \textbf{\color{blue}The Distribution of Quadratic Residues and non-residues.} 
%Mathematika, vol. 4, no. 2, 12/1957, pp. 106-112. MR0093504, %doi:10.1112/S0025579300001157.
%\bibitem{BD1962} Burgess, D. A. \textbf{\color{blue}On character sums and primitive roots}. Proc. London Math. Soc. (3) 12, 1962, 179-192.

%\bibitem{BD1963} Burgess, D. A. \textit{\color{red}A Note on the Distribution of Residues and Non-Residues}. Journal of the London Mathematical Society, vol. s1-38, no. 1, 1963, pp. 253--256. \href{https://mathscinet.ams.org/mathscinet-getitem?mr=0148628}{MR0148628}.  


%\bibitem{BG2013} Bober, J. Goldmakher, L. \textit{\color{red}Polya-Vinogradov and the Least Quadratic Nonresidue.} Mathematische Annalen, vol. 366, no. 1-2, 10/2016, pp. 853-863. http://arxiv.org/abs/1311.7556. \href{https://mathscinet.ams.org/mathscinet-getitem?mr=3552258}{MR3552258}.


%CCCCCCCCCCCCCCCCCCCCCCCCCCCCCCCCCCC

\bibitem{CN2021}Carella, N. \textit{\color{red}Upper Bound of the Least Quadratic Nonresidues.} \href{https://arxiv.org/pdf/2106.00544}{Arxiv.2106.00544}.

%\bibitem{CP2005} Crandall, Richard, Pomerance, Carl. \textbf{\color{blue} Prime Numbers A Computational Perspective.} Springer, New York.

%DDDDDDDDDDDDDDDDDDDDDDDDDDDDDDDD
\bibitem{DH2000} Davenport, H. \textit{\color{red}Multiplicative number theory.} Second Edition. Graduate Texts in Mathematics. Springer-Verlag, Berlin, 1980. \href{https://mathscinet.ams.org/mathscinet-getitem?mr=1790423}{MR1790423}.

%\bibitem{DI1995}Duke, W.; Friedlander, J. B.; Iwaniec, H. \textit{\color{red}Equidistribution of roots of a quadratic congruence to prime moduli.} Ann. of Math. (2) 141 (1995), no. 2, 423--441. \href{https://mathscinet.ams.org/mathscinet-getitem?mr=MR1324141}{MR1324141}.

\bibitem{DP2016} Dusart, P. \textit{\color{red}Estimates of some functions over primes, without R.H..} Math. Comp. 85 (2016), no. 298, 875--888. \href{http://arxiv.org/abs/1002.0442}{Arxiv.org/abs/1002.0442}.
\href{https://mathscinet.ams.org/mathscinet-getitem?mr=MR3434886}{MR3434886}.


%EEEEEEEEEEEEEEEEEEEEEEEEEEEEEEEEEEEEEEEEEEEEEEEEEE
%\bibitem{EP1961} Erdos, P. \textit{\color{red}Remarks on number theory. I.} Mat. Lapok 12 (1961) 10--17.
%\href{https://zbmath.org/0154.29403}{zbMath0154.29403}

%GGGGGGGGGGGGGGGGGGGGGGGGGGGGGGGG
%\bibitem{GF1801} Gauss, Friedrich. \textit{\color{red}Disquisitiones Arithmeticae}. Springer-Verlag, New York 1986. 
%\href{https://zbmath.org/0136.32301}{zbMath0136.32301}

\bibitem{GA2006} Gica, A. \textit{\color{red}Quadratic residues of certain types.} Rocky Mountain J. Math. 36 (2006), no. 6, 1867--1871.
\href{https://mathscinet.ams.org/mathscinet-getitem?mr=MR2305634}{MR2305634}.	

%\bibitem{GS2015} Granville, A., Soundararajan, K. \textbf{\color{blue}Large character sums: Burgess's theorem and zeros of $L$-functions}. http://arxiv.org/abs/1501.01804.

%\bibitem{GS2001} Granville, A., Soundararajan, K. \textbf{\color{blue}Upper bounds for $|L(1,\chi)|$.} Quarterly Journal of Mathematics, vol 53 (2002) pages 265-284. http://arxiv.org/abs/0106176.


%HHHHHHHHHHHHHHHHHHHHHHHHHHHHHHHHHHHHHHHHHHHHHHHHHHHHHHHHH

%\bibitem{HR1983}Hudson, R. H. \textit{\color{red}A note on prime $k$th power nonresidues.} Manuscripta Math. 42 (1983), no. 2-3, 285--288.
%\href{https://mathscinet.ams.org/mathscinet-getitem?mr=MR0701210}{MR0701210}.

%\bibitem{HR1980}Hudson, R. H.; Williams, K. S. \textit{\color{red}On the least quadratic nonresidue of a prime $p\equiv 3   \bmod   4$}.  J. Reine Angew. Math. 318 (1980), 106--109.
%\href{https://mathscinet.ams.org/mathscinet-getitem?mr=MR0579385}{MR0579385}.
  
%\bibitem{HR1976}Hudson, R. H. \textit{\color{red}A sharper bound for the least pair of consecutive $k$th power non-residues of non-principal characters  $\bmod   p$  of order $ k>3$.} Acta Arith. 30 (1976), no. 2, 133--135.
%\href{https://mathscinet.ams.org/mathscinet-getitem?mr=MR0422180}{MR0422180}.
			
\bibitem{HR1974}Hudson, R. H. \textit{\color{red}Power residues and nonresidues in arithmetic progressions.} Trans. Amer. Math. Soc. 194 (1974), 277--289. \href{https://mathscinet.ams.org/mathscinet-getitem?mr=MR0374002}{MR0374002}.	


%IIIIIIIIIIIIIIIIIIIIIIIIIIIIIIIIIIIIIIIIIIIIIIIIIIIIIII

\bibitem{IK2004} Iwaniec, H.; Kowalski, E. \textit{\color{red}Analytic number theory.} Amer. Math. Soc. Colloquium Publications, 53. American Mathematical Society, Providence, RI, 2004.
\href{https://mathscinet.ams.org/mathscinet-getitem?mr=2061214}{MR2061214}.	

%JJJJJJJJJJJJJJJJJJJJJJJJJJJJJJJJJJJJJJJJJJJJJJJJJJJJJJJJJJJJJJJJJJJJJJJJJ

\bibitem{JE1906}Jacobsthal, E. \textit{\color{red}Anwendungen einer Formel aus der Theorie der quadratischen Reste.} (German) JFM 37.0226.01 Berlin. 39 S. (1906). \href{https://zbmath.org/37.0226.01}{zbl.Math.37.0226.01}.

%KKKKKKKKKKKKKKKKKKKKKKKKKKKKKKK
%\bibitem{KB2017}Kerr, Bryce. \textbf{\color{blue}On certain exponential and character sums}. Thesis, 2017, UNSW.

%LLLLLLLLLLLLLLLLLLLLLLLLLLLLLLLLLLLL

\bibitem{LM2014} Lenstra, H. W., Jr.; Stevenhagen, P.; Moree, P. \textit{\color{red}Character sums for primitive root densities.} Math. Proc. Cambridge Philos. Soc. 157 (2014), no. 3, 489--511.
\href{https://mathscinet.ams.org/mathscinet-getitem?mr=MR3286520}{MR3286520}.

	\bibitem{LN1997} Lidl, R.; Niederreiter, H. \textit{\color{red}Finite fields.} Second edition. Encyclopedia of Mathematics and its Applications, 20. Cambridge University Press, Cambridge, 1997.
\href{https://mathscinet.ams.org/mathscinet-getitem?mr=MR1429394}{MR1429394}.

%\bibitem{LS2004} Louboutin, S. \textbf{\color{blue}Explicit upper bounds for $|L(1, \chi)|$ for primitive Dirichlet characters $\chi$}. Quart. J. Math 55 (2004), pp. 57-68.


%MMMMMMMMMMMMMMMMMMMMMMMMM

%\bibitem{MH1971} Montgomery, H. L. \textit{\color{red}Topics in multiplicative number theory.} Lecture Notes in Math., Vol. 227 Springer-Verlag, Berlin-New York, 1971. \href{https://mathscinet.ams.org/mathscinet-getitem?mr=MR0337847}{MR0337847}.

\bibitem{ML1972} Mordell, L. J. \textit{\color{red}On the exponential sum $\sum_{1\leq x \leq X} exp (2\pi i(ax+bg^x )/p)$.} Mathematika  19  (1972), 8---87. \href{https://zbmath.org/0245.10025}{zbl.Math.0245.10025}.

\bibitem{MV2007} Montgomery, H. L., Vaughan, R. C. \textit{\color{red}Multiplicative number theory. I}. Classical theory. Cambridge
University Press, Cambridge, 2007. \href{https://mathscinet.ams.org/mathscinet-getitem?mr=2378655}{MR2378655}.

%\bibitem{MV1977} Montgomery, H. L., Vaughan, R. C. \textit{\color{red}Exponential sums with multiplicative coefficients.} Invent. Math. 43 (1977), 6--82.
%\href{https://mathscinet.ams.org/mathscinet-getitem?mr=MR0457371}{MR0457371}.

\bibitem{MT2021} McGown, K.; Trevino, E. \textit{\color{red}The least quadratic non-residue.} Contemp. Math., 775, American Mathematical Society, RI, 2021, 20--231.
\href{https://mathscinet.ams.org/mathscinet-getitem?mr=4344311}{MR4344311}.

%\bibitem{MK2012} McGown, Kevin J. \textbf{\color{blue}On the Constant in Burgess' Bound for the Number of Consecutive Residues Or Non-Residues.} Functiones Et Approximatio Commentarii Mathematici, vol. 46, no. 2, 2012, pp. 273-284. MR2931671, doi:10.7169/facm/2012.46.2.10.


%PPPPPPPPPPPPPPPPPPPPPPPPPPPPPPPPPPPPPPP

\bibitem{PP2018}Pollack, P. \textit{\color{red}The least prime quadratic nonresidue in a prescribed residue class mod 4.} Journal of Number Theory 187 (2018) 403--414
\href{https://mathscinet.ams.org/mathscinet-getitem?mr=MR3766918}{MR3766918}.


\bibitem{PT2022} Antonella Perucca; Pietro Sgobba; Sebastiano Tronto
\textit{\color{red}Kummer theory for number fields via entanglement
	groups.} manuscripta math. 169, 251?270 (2022).
\href{https://mathscinet.ams.org/mathscinet-getitem?mr=3552258}{MR3552258?}.

% \bibitem{PR1932}Paley, R. E. \textit{\color{red}A theorem on characters.} J. London Math. Soc. 7 (1932), 28--32.
%\href{https://mathscinet.ams.org/mathscinet-getitem?mr=1574456}{MR1574456}.
%SSSSSSSSSSSSSSSSSSSSSSSSSSSSSSSSSSSSSSSSS

\bibitem{SP2003} Stevenhagen, P. \textit{\color{red}The correction factor in Artin's primitive root conjecture}. Theor. Nombres Bordeaux 15 (2003), no. 1, 383--391
\href{https://mathscinet.ams.org/mathscinet-getitem?mr=MR2019022}{MR2019022}.


%TTTTTTTTTTTTTTTTTTTTTTTTTTTTTTTTTTTTT

%\bibitem{TA2000} Toth, A. \textit{\color{red}Roots of quadratic congruences.} Internat. Math. Res. Notices(2000), no. 14, 719--739. \href{https://mathscinet.ams.org/mathscinet-getitem?mr=MR1776618}{MR1776618}.

%\bibitem{TE2015} Trevino, Enrique. \textbf{\color{blue}The Least $k$th th Power Non-Residue.} Journal of Number Theory, vol. 149, 04/2015, pp. 201-224. MR3296008, doi:10.1016/j.jnt.2014.10.019.

%\bibitem{TE2015} Trevino, Enrique. \textbf{\color{blue}The Burgess Inequality and the Least $K$th Power Non-Residue.} International Journal of Number Theory, vol. 11, no. 05, 08/2015, pp. 1653-1678. MR3376232, doi:10.1142/S1793042115400163.



%VVVVVVVVVVVVVVVVVVVVVVVVVVVVVVVV
%\bibitem{VI1918} Vinogradov, I.M. \textit{\color{red}Sur la distribution des residus et des non-residus des puissances}, Journal Physico-Math. Soc. Univ. Perm 1 (1918), 94--96. 

%\bibitem{VI1927} Vinogradov, I.M. \textbf{\color{red}On a general theorem concerning the distribution of the residues and non-residues of powers.} Trans. Amer. Math. Soc. 29 (1927), no. 1, 209--217. \href{https://mathscinet.ams.org/mathscinet-getitem?mr=1501384}{MR1501384}.

%WWWWWWWWWWWWWWWWWWWWWWWWWWWWWWWWWWWWWWWWW

\bibitem{WS2013}Wright, S. \textit{\color{red}Quadratic residues and non-residues in arithmetic progression.} Journal of Number Theory 133 (2013) 2398--2430.
\href{https://mathscinet.ams.org/mathscinet-getitem?mr=MR3035970}{MR3035970}.


\end{thebibliography}
\end{document}